\documentclass[a4paper,36pt,notitlepage]{article}
\usepackage{amsmath,amsthm,amssymb,mathrsfs}
\usepackage{enumerate}
\usepackage[dvipdfm]{pict2e}
\usepackage{graphicx}

\newtheorem{dfn}{Definition}[subsection]
\newtheorem{thm}[dfn]{Theorem}
\newtheorem{lem}[dfn]{Lemma}
\newtheorem{cor}[dfn]{Corollary}
\newtheorem{rem}[dfn]{Remark}
\newtheorem{prop}[dfn]{Proposition}\makeatletter

         \@addtoreset{equation}{section}
\makeatother

\begin{document}
\title{\bf {\large Two-sided random walks conditioned to have no intersections}}
\author{Daisuke Shiraishi}
\date{Research Institute for Mathematical Sciences \\ Kyoto University \\ \vspace{1.0\baselineskip} siraisi@kurims.kyoto-u.ac.jp}

\maketitle
\begin{abstract}
Let $S^{1},S^{2}$ be independent simple random walks in $\mathbb{Z}^{d}$ ($d=2,3$) started at the origin. We construct two-sided random walk paths conditioned that $S^{1}[0,\infty ) \cap S^{2}[1, \infty ) = \emptyset$ by showing the existence of the following limit:
\begin{equation*}
\lim _{n \rightarrow \infty } P ( \cdot \  | \ S^{1}[0, \tau ^{1} ( n) ] \cap S^{2}[1, \tau ^{2}(n) ] = \emptyset ),
\end{equation*}
where $\tau^{i}(n) = \inf \{ k \ge 0 : |S^{i} (k) | \ge n \}$. Moreover, we give upper bounds of the rate of the convergence. These are discrete analogues of results for Brownian motion obtained in \cite{Law*} and \cite{Law}.  
\end{abstract}

\section{Introduction and Main Results}
\subsection{Introduction}
Let $S=(S(n))$ be a simple random walk in $\mathbb{Z}^{d}$ ($d=2,3$) started at the origin. Take integers $k < n$. A time $k$ is called cut time up to $n$ if
\begin{equation}\label{cut time def}
S[0,k] \cap S[k+1, n] = \emptyset,
\end{equation}
where $S[0,k]= \{ S(j) : 0 \le j \le k \}$. We call $S(k)$ a cut point if $k$ is a cut time. Lawler \cite{Law2} has shown that there are constants $0 <c , c^{\prime} < \infty$ such that for all $n$,
\begin{equation}\label{no intersection probability}
c n^{- \frac{\xi}{2} } \le P ( S[0,n] \cap S[n+1, 2n] = \emptyset ) \le c^{\prime} n^{- \frac{\xi}{2} },
\end{equation}
where $\xi = \xi_{d}$ is the intersection exponent (see Section \ref{intersection exponent} below). Lawler, Schramm and Werner \cite{LSW} have proved that $\xi_{2}=\frac{5}{4}$ by using the SLE techniques. The value of $\xi_{3}$ is not still known. Let $J_{k}$ be the indicator function of the event that $k$ is a cut time up to $n$ and let $R_{n} = \textstyle\sum\limits _{k=0}^{n} J_{k}$. Lawler \cite{Law2} also proved that there exists $c>0$ such that
\begin{align*}
&P( R_{n} \ge c n^{ 1 - \frac{\xi}{2} } ) \ge c \ \text{ for } d=2, \\
&R_{n} \approx n^{ 1- \frac{\xi }{2} } \ \text{ with probability one for } d=3,
\end{align*}
where $\approx$ denotes that the logarithms of both sides are asymptotic.

While the understanding of the number of cut times has been advanced, there is a few results about the geometrical structure of the path around cut points, which is the purpose of this paper. We consider the following problem. If we condition that $S[0,n] \cap S[n+1, 2n] = \emptyset$, then what kind of structure does the path have around $S(n)$? Let $S^{1},S^{2}$ be independent simple random walks started at the origin. Then, thanks to the translation invariance and the reversibility of the simple random walk, our problem may be deduced to clarify the structure of $S^{1},S^{2}$ around the origin when we condition that $S^{1}[0,n] \cap S^{2}[1,n] = \emptyset$. Letting $n \rightarrow \infty$, we will face the following problems:
\begin{align}\label{problems}
&\text{(i) Construct two-sided path conditioned that } S^{1}[0, \infty ) \cap S^{2}[1, \infty ) = \emptyset. \\
&\text{(ii) What kind of geometrical structure does such a conditioned path have? } \\
&\text{(iii) Is the difference between two sided path conditioned } S^{1}[0,n] \cap S^{2}[1,n]=\emptyset \notag \\
&\text{ \ \ \ \ and the conditioned path in (i) small around the origin? } 
\end{align}

By \eqref{no intersection probability}, the probability that $S^{1}[0, \infty ) \cap S^{2}[1, \infty ) = \emptyset$ is 0 for $d=2,3$, so question (i) is not trivial. For Brownian motions, Lawler \cite{Law*}, and Lawler, Vermesi \cite{Law} have constructed Brownian paths conditioned to have no intersections. More precisely, let $B^{1},B^{2}$ be Brownian motions in $\mathbb{R}^{d}$ ($d=2,3$) starting distance one apart and 
\begin{equation*}
T^{i}(R) = \inf \{ t \ge 0 : |B^{i}(t)| = R \}.
\end{equation*}
In \cite{Law*}, it was proved that for $d=2$, the limit
\begin{equation}\label{Lawler results}
\lim _{n \rightarrow \infty } P( \cdot \ | \ B^{1}[0, T^{1}( e ^{n} ) ] \cap B^{2}[0, T^{2}( e^{n})] = \emptyset )
\end{equation}
exists and the rate of convergence is bounded above by $O( e^{- \delta \sqrt{n}})$ for some $\delta > 0$. For $d=3$, it was shown in \cite{Law} that the limit of \eqref{Lawler results} also exists and the rate of convergence is at most $O ( e^{ - \delta n})$ (see Proposition \ref{lawlerbmresult}). 

In this paper we will answer the question (i) and (iii). We will construct the path in (1.3) by proving the existence of the limit as in \eqref{Lawler results} for simple random walk (Theorem \ref{main theorem}). Furthermore, we will derive same rates of convergence as Brownian cases. Since the speed of convergence in Theorem \ref{main theorem} is relatively fast, it would give evidence that the gap considered in (1.5) is small.

Even though the conditioned Brownian paths were already constructed as in \eqref{Lawler results}, it is not straightforward to construct it for the simple random walk. Both in \cite{Law*} and \cite{Law}, the scaling property of Brownian motion is crucial in the construction and hence the same arguments cannot be applied for the simple random walk case. To overcome this problem, we will use the strong approximation of Brownian motion by simple random walk derived from the Skorohod embedding. By this approximation, we can define simple random walks $S^{1},S^{2}$ and Brownian motions $B^{1}, B^{2}$ on the same probability space so that with high probability, the paths of $S^{i}$ are very close to those of $B^{i}$. However, if $S^{1}$ and $S^{2}$ start from a same point, then the difference between the path of $S^{i}$ and that of $B^{i}$ is too large to control the difference between $P(B^{1}[0,n] \cap B^{2}[1,n] = \emptyset )$ and $P( S^{1}[0,n] \cap S^{2}[1,n] = \emptyset ) $. (See Proposition \ref{skorohod} for the difference between $S^{i}[0,n]$ and $B^{i}[0,n]$. We must admit the fact that the difference may be of order $n^{\frac{1}{4}}$.) This difficulty can be dealt with using the following ideas. Even if starting points of $S^{1}$ and $S^{2}$ are very close, they gradually have a good chance of being reasonably far apart because of the conditioning not to intersect. Once $S^{1}$ and $S^{2}$ are far apart, we can use the Skorohod embedding to control the non-intersection probability of simple random walks (see Proposition \ref{upper-bound} for details).  

The question (iii) will be discussed in a forthcoming paper \cite{S}. Let $\overline{S}^{1}, \overline{S}^{2}$ be the associated two-sided random walks whose probability law is $P^{\sharp}$ in Theorem \ref{main theorem}. In order to show that paths of $\overline{S}^{i}$ have different structures from those of usual simple random walk $S^{i}$, we will consider a simple random walk on $\overline{\cal G} := \overline{S}^{1}[0,\infty ) \cup \overline{S}^{2}[0,\infty )$. (Here we regard $\overline{\cal G}$ as the subgraph consisting of all the vertices visited and edges traversed by either $\overline{S}^{1}$ or $\overline{S}^{2}$.) In \cite{S}, it will be shown that the simple random walk on $\overline{\cal G}$, say $X$, has subdiffusive behavior for $d=2$. This is due to that $\overline{\cal G}$ has many so called bottleneck edges and it takes much longer for $X$ to move away from its starting point compared to the simple random walk in $\mathbb{Z}^{2}$. 

Throughout this paper, we use $c, c^{\prime}, c_{1} , c_{2}, \cdots $ to denote arbitrary constants that depend only on the dimension $d$. The values of them may change from place to place.





\subsection{Framework and Main results}
Let $d=2,3$. For $x \in \mathbb{Z}^{d}$, let
\begin{equation*}
{\cal B}(x, n) = \{ z \in \mathbb{Z}^{d} : |z| < n \}
\end{equation*}
and
\begin{equation*}
\partial {\cal B}(x,n) = \{ z \in \mathbb{Z}^{d} \backslash {\cal B}(x, n) : |z-y| =1 \text{ for some } y \in {\cal B}(x, n) \}.
\end{equation*}
We write ${\cal B} (n) ={\cal B}(0,n)$ and $\partial {\cal B}(n) = \partial {\cal B}(0,n)$. Let ${\cal B}_{k}(x) = {\cal B}(x, 2^{k})$ and $\partial {\cal B}_{k}(x) = \partial {\cal B} (x, 2^{k})$. We also write ${\cal B}_{k} = {\cal B}_{k}(0)$ and $\partial {\cal B}_{k} = \partial {\cal B}_{k}(0)$.

A sequence of points $\gamma = [ \gamma(0) , \gamma (1), \cdots , \gamma (l)] \subset \mathbb{Z}^{d}$ is called path if $|\gamma(j) - \gamma(j-1)|=1$ for each $j=1,2, \cdots , l$. We let len$\gamma=l$ be the length of the path, $\Lambda (n)$ be the set of paths satisfying that
\begin{align*}
&\gamma (0)=0, \gamma (j) \in {\cal B}(n) \text{ for all } j=0,1, \cdots , \text{len}\gamma -1 \\
&\gamma( \text{len}\gamma ) \in \partial {\cal B}(n). 
\end{align*}
Let
\begin{equation*}
\Gamma (n) = \{ \overline{\gamma}=(\gamma ^{1}, \gamma ^{2}) \in \Lambda (n) ^{2} : \gamma^{1}(i) \neq \gamma ^{2}(j) \text{ for all } (i,j) \neq (0,0) \}, 
\end{equation*}
and $\Gamma (\infty ) = \bigcap _{n=1}^{\infty} \Gamma (n)$. 
We write $\Gamma_{k}= \Gamma(2^{k})$.

Let $S^{1},S^{2}$ be the independent simple random walks in $\mathbb{Z}^{d}$ started at the origin. Let
\begin{equation*}
\tau ^{i}(n) =\inf \{ k \ge 0 : S^{i}(k) \in \partial {\cal B} (n) \},
\end{equation*}
and $\tau ^{i}_{k}= \tau ^{i}(2^{k})$.


\begin{thm}\label{main theorem}
Let $d=2$ or $3$. For each $L$ and $\overline{\gamma} \in \Gamma (L)$, the limit 
\begin{equation}\label{shuusoku}
\lim _{N \rightarrow \infty } P \Big( (S^{1}[0,\tau ^{1}( L ) ] , S^{2}[0,\tau ^{2}( L ) ] )= \overline{\gamma} \  \big| \  (S^{1}[0,\tau ^{1}( N) ] , S^{2}[0,\tau ^{2}( N ) ] ) \in \Gamma (N) \Big) =: P^{\sharp}(\overline{\gamma})
\end{equation}
exists. Furthermore, there exist $\delta > 0$ and $c < \infty$ depending only on the dimension such that the following holds for all $L$ and $\overline{\gamma} \in \Gamma (L)$.
\begin{align}\label{shuusokunohayasa}
&\Big| P \Big( (S^{1}[0,\tau ^{1}( L ) ] , S^{2}[0,\tau ^{2}( L ) ] )= \overline{\gamma} \  \big| \  (S^{1}[0,\tau ^{1}( N) ] , S^{2}[0,\tau ^{2}( N ) ] ) \in \Gamma (N) \Big) - P^{\sharp}(\overline{\gamma}) \Big| \le c e^{ - \delta \sqrt{\log N}} \\
&\text{for } d=2, \notag \\
&\Big| P \Big( (S^{1}[0,\tau ^{1}( L ) ] , S^{2}[0,\tau ^{2}( L ) ] )= \overline{\gamma} \  \big| \  (S^{1}[0,\tau ^{1}( N) ] , S^{2}[0,\tau ^{2}( N ) ] ) \in \Gamma (N) \Big) - P^{\sharp}(\overline{\gamma}) \Big| \le c N^{-\delta } \\
&\text{for } d=3, \notag
\end{align}
and $P^{\sharp}$ extends uniquely to a probability measure on $\Gamma (\infty )$.

\end{thm}

The paper is organized as follows. Section 2 gives some preliminary propositions about Brownian motions and simple random walks. In particular, we state the Skorohod embedding which is crucial in this paper. Key estimates are given in Section 3 by using this approximation. We give the proof of Theorem \ref{main theorem} in Section 4.

\section{Known Results}

In this section, we give a list of definition of the objects and known results commonly used throughout this paper. 

\subsection{Intersection Exponent}\label{intersection exponent}
In this subsection, we review the intersection exponent for Brownian motion and simple random walk. Let $d=2$ or $3$. Let $B^{1},B^{2}$ be independent Brownian motions in $\mathbb{R}^{d}$. We start by stating the estimate from \cite{Law+}. Let
\begin{equation*}
T^{i}(n) = \inf \{ t \ge 0: |B^{i}(t)| =n \} ,
\end{equation*}
and write $P^{x,y}=P^{x,y}_{1,2}$ to denote probabilities assuming $B^{1}(0)=x, B^{2}(0)=y$. Then we have the following proposition.
\begin{prop}\label{intersection exp bm} (\cite{Law+}, Corollary 3.13.)
There exist $\xi = \xi _{d}$, $ c < \infty$ and an increasing function $f : (0,2] \rightarrow (0,\infty )$ such that if $|x|=|y|=1$, then for all $n \ge 1$
\begin{equation}\label{probab estimate no int}
f(|x-y|) n^{-\xi} \le P^{x,y}( B^{1}[0,T^{1}(n)] \cap B^{2}[0,T^{2}(n)] = \emptyset ) \le c n^{-\xi}.
\end{equation}
\end{prop} 

Next we state the analogues for simple random walks. Let $S^{1}, S^{2}$ be independent simple random walks in $\mathbb{Z}^{d}$. Again we write $P^{x,y}=P^{x,y}_{1,2}$ to denote probabilities assuming $S^{1}(0)=x, S^{2}(0)=y$. Let
\begin{equation*}
\tau ^{i}(n) = \inf \{ k \ge 0 : |S^{i}(k)| \ge n \}.
\end{equation*}
Then the following proposition was proved in \cite{Law2}.
\begin{prop}\label{intersection exp rw}(\cite{Law2}, Theorem 1.3, Corollary 4.6.)
Let $\xi$ be the exponent in Proposition \ref{intersection exp bm}. Then there exist constants $c_{1},c_{2}$ such that the following holds.
\begin{align}\label{probab estimate no int rw}
c_{1} n^{-\xi } \le &P^{0,0}(S^{1}[0,\tau ^{1}(n)] \cap S^{2}(0, \tau ^{2}(n)] = \emptyset ) \le c_{2} n^{-\xi}, \\
\sup _{|x| , |y| \le m} &P^{x,y}( S^{1}[0,\tau ^{1}(n)] \cap S^{2}(0, \tau ^{2}(n)] = \emptyset ) \le c_{2} \big( \frac{n}{m} \big) ^{-\xi}, 
\end{align}
for all $m \le n$.
\end{prop}

\begin{rem}\label{values}
In \cite{LSW}, it was proved that
\begin{equation}\label{2-dim}
\xi_{2}= \frac{5}{4}.
\end{equation}
The value of $\xi _{3}$ is not known. Rigorous estimate (\cite{BL}, \cite{Law+}) show that $\frac{1}{2} < \xi _{3} < 1$. Simulations suggests that $\xi _{3}$ is around $0.57$ (see Section 7 in \cite{Law}).
\end{rem}

\subsection{Skorohod Embedding}
In this subsection, we state the strong approximation of Brownian motion by simple random walk derived from the Skorohod embedding (see \cite{Law2} for details).

\begin{prop}\label{skorohod} (\cite{Law2}, Lemma 3.1, Lemma 3.2.)
There exist a probability space $(\Omega , {\cal F} , P )$ containing a $d$-dimensional standard Brownian $B$ and $d$-dimensional simple random walk $S$ such that the following holds. For every $\epsilon > 0$ there exist $\delta > 0$ and $ a < \infty$ such that 
\begin{equation}\label{approx time}
P \big( \sup_{0 \le t \le n} |B(t) - S(td) | \ge n^{\frac{1}{4} + \epsilon } \big) \le a \exp ( - n^{\delta} ).
\end{equation}
Moreover, if we set
\begin{equation*}
T(n) = \inf \{ t : |B(t)|=n \}, \ \ \ \  \tau ( n) = \inf \{ j : |S(j)| \ge n \}
\end{equation*}
then for every $\epsilon > 0$ there exist $\delta > 0$ and $ a < \infty$ such that 
\begin{equation}\label{approx stop}
P \big( \sup_{0 \le t \le T(n)} |B(t) - S(td) | \ge n^{\frac{1}{2} + \epsilon } \big) \le a \exp ( - n^{\delta} ).
\end{equation}  
\end{prop}

We will be using the strong Markov property at time $T(n)$. However, one slight complication that arises is the fact that $\{ B(t) , S(td) : t \le T(n) \}$ might contain a little information about $B(t)$ beyond time $T(n)$. To overcome this problem, we need the following proposition.

\begin{prop}\label{measurability} (\cite{Law2}, Lemma 3.3.)
There exist $\delta > 0$ and $a < \infty$ such that the following holds. For each $n$, there is an event $\Psi (n)$ with 
\begin{equation*}
P(\Psi (n) ) \ge 1-a \exp (- n^{\delta })
\end{equation*}
such that on the event $\Psi (n)$,
\begin{equation*}
\{ B(t) : t \le \max \{ T(n) , \tau (n) \} \} \cup \{ S(td) : t \le \max \{ T(n) , \tau (n) \} \}
\end{equation*}
and
\begin{equation*}
\{ B(t) : t \ge T(2n) \}
\end{equation*}
are conditionally independent given $B(T(2n))$.
\end{prop}

\subsection{Beurling Estimate}
We need some estimates that say intuitively two random walks that get close each other are very likely intersect. For $d=2$, it is a case of the Beurling estimate. For $d=3$, corresponding estimates were obtained in \cite{Law2}. Here we state them.

Let $B$ be the Brownian motion in $\mathbb{R}^{2}$ and $S$ be the simple random walk in $\mathbb{Z}^{2}$. Then the following are well-known (see \cite{Law3} for the continuous case and \cite{L1} for discrete case).

\begin{prop}\label{beurling estimate}
(i) (\cite{Law3}, Theorem 3.76) There exists a constant $K < \infty$ such that for any $R \ge 1$, any $x \in \mathbb{R}^{2}$ with $|x| \le R$, any $A \subset \mathbb{R}^{2}$ with $[0,R] \subset \{ |z| : z \in A \}$,
\begin{equation}\label{bmbeurlingestimate}
P^{x} ( T(R) < T_{A} ) \le K \big( \frac{|x|}{R} \big) ^{\frac{1}{2}},
\end{equation}
where $T(R) = \inf \{ t \ge 0 : |B(t)| \ge R \}$ and $T_{A} = \inf \{ t \ge 0 : B(t) \in A \}$.  \\
(ii) (\cite{L1}, Theorem 2.5.2.) There exists a constant $K < \infty$ such that for any $n \ge 1$, any $x \in \mathbb{Z}^{2}$ with $|x| \le n$, any connected set $A \subset \mathbb{Z}^{2}$ containing the origin and such that $\sup \{ |z| : z \in A \} \ge n$,
\begin{equation}\label{srwbeurlingestimate}
P^{x} ( \tau (n) < \tau _{A} ) \le K \big( \frac{|x|}{n} \big) ^{\frac{1}{2}},
\end{equation}
where $\tau (n) = \inf \{ j \ge 0 : |S(j)| \ge n \}$ and $\tau _{A} = \inf \{ j \ge 0 : S(j) \in A \}$.

\end{prop} 

For $d=3$, there is no useful analogue of Proposition \ref{beurling estimate}. So we need some more work. Let $B, B^{\prime}$ be two independent Brownian motion in $\mathbb{R}^{3}$. For each $\epsilon > 0$ and $b < \infty$, let
\begin{equation*}
Z_{n}= Z_{n}(\epsilon , b) = \sup P^{z} ( B[ 0, T(2n) ] \cap B^{\prime} [0, T^{\prime}(2n) ] = \emptyset \ | \ B^{\prime} [0, T^{\prime}(2n) ] ),
\end{equation*}
where the supremum is over all $z$ with $|z| \le n$ such that
\begin{equation*}
\text{dist} (z, B^{\prime} [0, T^{\prime}(2n) ] ) \le b n^{1-\epsilon}, 
\end{equation*}
and $T(n)$ (resp. $T^{\prime}(n)$) be the first hitting time of $B$ (resp. $B^{\prime}$) to the boundary of disk centered at the origin with radius $n$.
Note that $P^{z}$ denotes the probability with $B(0)=z$ and $Z_{n}$ is a function of $B^{\prime} [0, T^{\prime}(2n) ]$. The following proposition says that Brownian path is a `hittable set' with high probability.

\begin{prop}\label{hittable for bm} (\cite{Law2}, Lemma 2.4.)
For every $M < \infty , \epsilon > 0 , b < \infty$, there exist $\delta > 0$ and $a < \infty$ such that for $|x| \le n$,
\begin{equation}\label{d=3hittablepath}
P^{\prime x}(Z_{n} \ge n^{-\delta } ) \le a n^{- M},
\end{equation}
where $P^{\prime x}$ denotes probability with $B^{\prime}(0) = x$.
\end{prop}

Finally, we state an analogue of this proposition for simple random walks. Let $S,S^{\prime}$ be two independent simple random walks in $\mathbb{Z}^{3}$. For each $\epsilon > 0$ and $b < \infty $, let 
\begin{equation*}
Z_{n}^{\sharp}= Z_{n}(\epsilon , b)^{\sharp} = \sup P^{z} ( S[ 0, \tau (2n) ] \cap S^{\prime} [0, \tau ^{\prime}(2n) ] = \emptyset \ | \ S^{\prime} [0, \tau ^{\prime}(2n) ] ),
\end{equation*}
where the supremum is over all $z$ with $|z| \le n$ and 
\begin{equation*}
\text{dist} (z, S^{\prime} [0, \tau ^{\prime}(2n) ] ) \le b n^{1-\epsilon}, 
\end{equation*}
and $\tau (n)$ (resp. $\tau ^{\prime}(n)$) be the first hitting time of $S$ (resp. $S^{\prime}$) to $\partial {\cal B}(n)$.
Again note that $P^{z}$ denotes the probability with $S(0)=z$ and $Z_{n}^{\sharp}$ is a function of $S^{\prime} [0, \tau ^{\prime}(2n) ]$. Then we have the following.

\begin{prop}\label{hittable for srw} (\cite{Law2}, Lemma 2.6.)
For every $M < \infty , \epsilon > 0 , b < \infty$, there exist $\delta > 0$ and $a < \infty$ such that for $|x| \le n$,
\begin{equation}\label{d=3hittablerwpath}
P^{\prime x}(Z_{n}^{\sharp} \ge n^{-\delta } ) \le a n^{- M},
\end{equation}
where $P^{\prime x}$ denotes probability with $S^{\prime}(0) = x$.
\end{prop}

\subsection{Nonintersecting Brownian motions}
In this subsection, we state convergence theorems for Brownian motion in $\mathbb{R}^{2}$ and $\mathbb{R}^{3}$ obtained in \cite{Law*} and \cite{Law}, respectively. Let $d=2$ or $3$, and $B^{1},B^{2}$ be independent Brownian motions in $\mathbb{R}^{d}$. Let ${\cal D}= \{ z \in \mathbb{R}^{d} : |z| \le 1 \}$ and $ \partial {\cal D} = \{ z \in \mathbb{R}^{d} : |z|=1 \} $. For $K_{1},K_{2} \subset {\cal D}$ and $w=(w_{1},w_{2}) \in \partial {\cal D}^{2}$ with $w_{j} \in K_{j} \cap \partial {\cal D}$, define
\begin{equation*}
A_{n}(K_{1},K_{2}) = \{ B^{1}[0,T^{1}(e^{n})] \cap B^{2}[0,T^{2}(e^{n})]=\emptyset , B^{1}[0,T^{1}(e^{n})] \cap K_{2} = \emptyset , B^{2}[0,T^{2}(e^{n})] \cap K_{1}=\emptyset \},
\end{equation*}
where $T^{i}(R)= \inf \{ t \ge 0 : |B^{i}(t) | \ge R \}$.
Let 
\begin{equation*}
Q_{n}(K ,w ) = e^{n\xi}P^{w_{1},w_{2}}( A_{n}(K_{1},K_{2}) ).
\end{equation*}
Here $\xi=\xi_{d}$ is the intersection exponent defined as in Section \ref{intersection exponent}.
In \cite{Law*} and \cite{Law}, it was shown the following convergence theorems for $d=2$ and $d=3$, respectively.
\begin{prop}\label{lawlerbmresult} (\cite{Law*}, Theorem 1.2 and \cite{Law}, Proposition 4.8.)
Let $d=2$ or $3$. For each $K_{1},K_{2} \subset {\cal D}$ and $w=(w_{1},w_{2}) \in \partial {\cal D}^{2}$ with $w_{j} \in K_{j} \cap \partial {\cal D}$, the limit
\begin{equation}\label{convergencethm}
\lim _{n \rightarrow \infty} Q_{n}(K ,w ) =: Q(K,w)
\end{equation}
exists. Moreover there exist $c < \infty$ and $\beta > 0$ depending only on the dimension such that the following holds.
\begin{align}\label{speedofconvergence}
|Q(K,w) - Q_{n}(K ,w )| \le c e^{-\beta \sqrt{n}} Q(K,w) \ \ \ \ \text{ for } d=2, \\
|Q(K,w) - Q_{n}(K ,w )| \le c e^{-\beta n} Q(K,w) \ \ \ \ \ \ \text{ for } d=3. 
\end{align}
\end{prop}
As mentioned, our main result Theorem \ref{main theorem} (or Theorem \ref{cauchy seq} below) is a random walk version of this proposition. Notice that the rate of convergence in Theorem \ref{main theorem} is same as that of Proposition \ref{lawlerbmresult}.

\section{Approximation of non-intersection probabilities}

\subsection{Preliminary}

Fix $L \in \mathbb{N}$ and $ \overline {\gamma}=(\gamma^{1}, \gamma^{2}) \in \Gamma _{L}$. We write $w^{i} = \gamma ^{i} ( \text{len} \gamma^{i})$ for the end point of $\gamma^{i}$. Assume $10L < m<n$. Let $S^{1}, S^{2}$ be two independent simple random walks in $\mathbb{Z}^{d}$ starting at $w^{1}, w^{2}$ respectively. Let $A_{m} ( \overline {\gamma} )$ denote the event
\begin{eqnarray}\label{non intersect} 
A_{m} ( \overline {\gamma} )=\left\{ \begin{array}{ll}
S^{1}[0,\tau ^{1}_{m}] \cap \gamma^{2} = \emptyset , \\
S^{2}[0,\tau ^{2}_{m}] \cap \gamma^{1} = \emptyset , \\
S^{1}[0,\tau ^{1}_{m}] \cap S^{2}[0,\tau ^{2}_{m}] = \emptyset  
\end{array}
\right\}.
\end{eqnarray}
The goal of this section is to prove the following proposition.

\begin{prop}\label{goal}
Let $d=2,3$. For all $L \in \mathbb{N}$ and $\overline {\gamma}=(\gamma^{1}, \gamma^{2}) \in \Gamma _{L}$, there exist $c< \infty$ and $\delta > 0$ such that for all $n> m> 10 L$,
\begin{equation}\label{cauchy}
|2^{(m-L) \xi} P ( A_{m} ( \overline {\gamma} ) ) - 2^{(n-L) \xi} P ( A_{n} ( \overline {\gamma} ) ) | \le c 2^{-\delta m^{\frac{d}{2}-\frac{1}{2}} }.
\end{equation}
\end{prop}

\subsection{Several Lemmas}

For $\frac{m}{3} \le j \le \frac{m}{2}$, let
\begin{equation}\label{dj}
D_{j}= \min \{ \text{dist} ( S^{1}( \tau ^{1}_{j} ) , S^{2}[0, \tau ^{2}_{j}] ), \text{dist} ( S^{2}( \tau ^{2}_{j}) , S^{1}[0, \tau ^{1}_{j}] ) \}
\end{equation}

\begin{lem}\label{djlarge}
There exist $c < \infty$ and $\delta > 0$ such that for all $N \ge m$,
\begin{equation}\label{0.99}
P ( A_{N} (\overline {\gamma}) , D_{j} \le 2^{0.99j} ) \le c 2^{-(N-L)\xi} 2^{-\delta j},
\end{equation}
for each $\frac{m}{3} \le j \le \frac{m}{2}$.
\end{lem}

\begin{proof}
It is enough to show that 
\begin{equation}\label{0.99i}
P( A_{N} (\overline {\gamma}), \text{dist} ( S^{1}( \tau ^{1}_{j}) , S^{2}[0, \tau ^{2}_{j}] )  \le 2^{0.99j} ) \le  c 2^{-(N-L)\xi} 2^{-\delta j}.
\end{equation}
By the strong Markov property,
\begin{equation*}
P( A_{N} (\overline {\gamma}), \text{dist} ( S^{1}( \tau ^{1}_{j}) , S^{2}[0, \tau ^{2}_{j}] ) \le 2^{0.99j} ) \le c 2^{-(N-j-1)\xi} P( A_{j+1} (\overline {\gamma}), \text{dist} ( S^{1}( \tau ^{1}_{j}) , S^{2}[0, \tau ^{2}_{j}] ) \le 2^{0.99j} ).
\end{equation*}
Applying Proposition \ref{hittable for srw} with $\epsilon = 0.01$, $b=1$, $S=S^{1}$ and $S^{\prime}=S^{2}$, we see that there exist $\delta > 0$ and $c < \infty$ such that
\begin{align*}
 &P( A_{j+1} (\overline {\gamma}), \text{dist} ( S^{1}( \tau ^{1}_{j}) , S^{2}[0, \tau ^{2}_{j}] ) \le 2^{0.99j})\\
 &\le P (  Z^{\sharp}_{2^{j}}(0.01, 1) \ge 2^{-\delta j} ) + P( A_{j+1} (\overline {\gamma}), \text{dist} ( S^{1}( \tau ^{1}_{j}) , S^{2}[0, \tau ^{2}_{j}] ) \le 2^{0.99j}, Z^{\sharp}_{2^{j}}(0.01, 1) \le 2^{-\delta j}) \\
 &\le c 2^{-6j} + P( A_{j+1} (\overline {\gamma}), \text{dist} ( S^{1}( \tau ^{1}_{j}) , S^{2}[0, \tau ^{2}_{j}] ) \le 2^{0.99j}, Z^{\sharp}_{2^{j}}(0.01, 1) \le 2^{-\delta j}).
 \end{align*}
 By the strong Markov property,
 \begin{equation*}
  P( A_{j+1} (\overline {\gamma}), \text{dist} ( S^{1}( \tau ^{1}_{j}) , S^{2}[0, \tau ^{2}_{j}] ) \le 2^{0.99j}, Z^{\sharp}_{2^{j}}(0.01, 1) \le 2^{-\delta j}) \le 2^{-\delta j} P( A_{j} (\overline {\gamma}) ).
  \end{equation*}
  Since $P( A_{j} (\overline {\gamma}) ) \le c 2^{-(j-L)\xi}$, the lemma is finished.

\end{proof}

Let 
\begin{equation}\label{eventdj}
F_{m}=\{ D_{\frac{m}{3}} \ge 2^{\frac{0.99m}{3}} \}.
\end{equation}
By Lemma \ref{djlarge}, there exists $\delta > 0$ such that $P ( A_{N} (\overline {\gamma}) , F_{m}^{c} ) \le c 2^{-(N-L)\xi} 2^{-\delta m}$ for every $N \ge m$.

For each $i=1,2$, define
\begin{equation*}
\sigma ^{i}= \sigma ^{i}_{m}= \inf \{ k \ge \tau ^{i} (2^{\frac{m}{3}}-2^{\frac{2m}{9}}) : S^{i}(k) \in \partial {\cal B} ( S^{i}(\tau ^{i} (2^{\frac{m}{3}}-2^{\frac{2m}{9}})) , 2^{\frac{m}{4}}) \}.
\end{equation*}

\begin{lem}\label{boundaryhit}
There exist $\delta > 0$ and $c < \infty$ such that for each $N \ge m$,
\begin{equation}\label{hitproba}
P( A_{N} (\overline {\gamma}) , \sigma ^{i} < \tau ^{i} _{\frac{m}{3}} ) \le c 2^{-(N-L)\xi} 2^{-\delta m}.
\end{equation}
\end{lem}

\begin{proof}
By the strong Markov property,
\begin{equation*}
P( A_{N} (\overline {\gamma}) , \sigma ^{i} < \tau ^{i} _{\frac{m}{3}} ) \le c 2^{-(N-\frac{m}{3})\xi} P( A_{\frac{m}{3}} (\overline {\gamma}) , \sigma ^{i} < \tau ^{i} _{\frac{m}{3}} ).
\end{equation*}
Since $\sigma ^{i} > \tau ^{i} _{\frac{m}{3}-1}$, we see that
\begin{equation*}
P( A_{\frac{m}{3}} (\overline {\gamma}) , \sigma ^{i} < \tau ^{i} _{\frac{m}{3}} ) \le E_{3-i} \big (E_{i} \big( {\bf 1} \{ A_{\frac{m}{3}-1} (\overline {\gamma})\} P_{i}^{S^{i}(\tau ^{i} (2^{\frac{m}{3}}-2^{\frac{2m}{9}}))}(\sigma ^{i} < \tau ^{i} _{\frac{m}{3}}) \big) \big).
\end{equation*}
It is easy to see that there exist $ \delta > 0$ and $ c < \infty$ such that
\begin{equation*}
P_{i}^{S^{i}(\tau ^{i} (2^{\frac{m}{3}}-2^{\frac{2m}{9}}))}(\sigma ^{i} < \tau ^{i} _{\frac{m}{3}}) \le c 2^{-\delta m},
\end{equation*}
and the lemma is proved.

\end{proof}

Let $G_{m}= \{ S^{i}[\tau ^{i} (2^{\frac{m}{3}}-2^{\frac{2m}{9}}), \tau ^{i} _{\frac{m}{3}}] \subset {\cal B} ( S^{i}(\tau ^{i} (2^{\frac{m}{3}}-2^{\frac{2m}{9}})), 2^{\frac{m}{4}}), \text{ for } i=1,2 \}$ From Lemma \ref{boundaryhit}, we have
\begin{equation}\label{gm}
P( A_{N} (\overline {\gamma}), G_{m}^{c} ) \le c 2^{-(N-L)\xi} 2^{-\delta m}.
\end{equation}

Finally, let 
\begin{equation}\label{gosa}
Z^{i}_{m}= \sup P^{z}( S[0,\tau_{ \frac{m}{3} } ] \cap S^{i}[0, \tau_{ \frac{m}{3} }] = \emptyset ),
\end{equation}
where the supremum is over all $z$ with 
\begin{equation*}
\text{dist} (z, S^{i}[0, \tau ^{i} (2^{\frac{m}{3}}-2^{\frac{2m}{9}}) ] ) \le 2^{\frac{11m}{60}}. 
\end{equation*}
Note that $Z^{i}_{m}$ is a function of $S^{i}[0, \tau_{ \frac{m}{3} }]$. By Proposition \ref{hittable for srw}, we see that there exist $ \delta >0$ and $c < \infty$ such that
\begin{equation}\label{hittablepath}
P_{i}( Z^{i}_{m} \ge 2^{-\delta m} ) \le c 2^{-6m}.
\end{equation}
Therefore, if we set $H^{i}_{m} = \{ Z^{i}_{m} < 2^{-\delta m} \}$ and $H_{m} = H^{1}_{m} \cap H^{2}_{m}$, we have
\begin{equation}\label{hm}
P( A_{N} (\overline {\gamma}), H_{m}^{c} ) \le c 2^{-(N-L)\xi} 2^{-\delta m}.
\end{equation}

\subsection{Coupling}
Using the strong Markov property, we see that
\begin{equation}\label{stmarkov}
P( A_{N} (\overline {\gamma}), F_{m} , G_{m} , H_{m} )  
= E\big( {\bf 1} \{  A_{\frac{m}{3}} (\overline {\gamma}) , F_{m} , G_{m} , H_{m} \} P^{S^{1}(\tau^{1}_{ \frac{m}{3} }), S^{2}(\tau^{2}_{ \frac{m}{3} }) 
}_{1,2} ( R_{\frac{m}{3} , N} ) \big),
\end{equation}
where we denote $R_{\frac{m}{3} , N}$ be the event
\begin{eqnarray}
R_{\frac{m}{3} , N}= \left\{
\begin{array}{ll}
\overline{S}^{1}[0, \tau^{1}_{N}] \cap ( S^{2}[0,\tau^{2}_{ \frac{m}{3} }] \cup \gamma^{2} ) = \emptyset    \\
\overline{S}^{2}[0, \tau^{2}_{N}] \cap ( S^{1}[0,\tau^{1}_{ \frac{m}{3} }] \cup \gamma^{1} ) = \emptyset    \\
\overline{S}^{1}[0, \tau^{1}_{N}] \cap \overline{S}^{2}[0, \tau^{2}_{N}] = \emptyset
\end{array}
\right\}.
\end{eqnarray} 
Here $\overline{S}^{1}$ and $\overline{S}^{2}$ are independent simple random walks starting at $S^{1}(\tau^{1}_{ \frac{m}{3} })$ and $S^{2}(\tau^{2}_{ \frac{m}{3} })$, respectively, and we use same notation $\tau ^{i} (R) , \tau ^{i}_{k}$ for the hitting time of $\overline{S}^{i}$. More precisely, let
\begin{equation*}
\tau^{i} (R) =\inf \{ j \ge 0 : \overline{S}^{i} \in \partial {\cal B}(R) \}
\end{equation*}
and $\tau ^{i}_{k} =\tau ^{i}(2^{k})$.
Throughout this section we will let $(B^{1}, \overline{S}^{1})$ and $(B^{2}, \overline{S}^{2})$ be two independent Brownian motion - random walk
pairs coupled as in Section 2.2. Assume $B^{i}(0)=\overline{S}^{i}(0)= S^{i}(\tau^{i}_{ \frac{m}{3} })=:w^{i}_{m/3}$. Let
\begin{equation}\label{bmtime}
T^{i}(R)= \inf \{ t \ge 0 : |B^{i}| =R  \},
\end{equation}
and $T^{i}_{k} = T^{i}(2^{k})$.
From now on, we assume the event $A_{\frac{m}{3}} (\overline {\gamma}) \cap F_{m} \cap G_{m} \cap H_{m}$ holds and compare the probability that two Brownian motions do not intersect each other with the probability that simple random walks do not intersect. For this purpose, let
\begin{equation}\label{fatpath}
\text{PATH}^{i}_{f}=\text{PATH}^{i}_{f,m}= \{ z\in \mathbb{R}^{d} : \text{dist}( z, S^{i}[0,\tau^{i}_{ \frac{m}{3} }] \cup \gamma^{i} ) \le 2^{\frac{11m}{60}} \}.
\end{equation}
be a fattened path of $S^{i}[0,\tau^{i}_{ \frac{m}{3} }] \cup \gamma^{i}$. Set
\begin{eqnarray}\label{firstintrw}
\beta = \inf \left\{ \frac{m}{3} \le k \le N :
\begin{array}{ll}
\overline{S}^{1}[\tau ^{1}_{k}, \tau^{1}_{N}] \cap ( \overline{S}^{2}[0, \tau ^{2}_{k}]  \cup S^{2}[0,\tau^{2}_{ \frac{m}{3} }] \cup \gamma^{2} ) = \emptyset    \\
\overline{S}^{2}[\tau ^{2}_{k}, \tau^{2}_{N}] \cap ( \overline{S}^{1}[0, \tau ^{1}_{k}]  \cup S^{1}[0,\tau^{1}_{ \frac{m}{3} }] \cup \gamma^{1} ) = \emptyset    \\
\overline{S}^{1}[\tau ^{1}_{k}, \tau^{1}_{N}] \cap \overline{S}^{2}[\tau ^{2}_{k}, \tau^{2}_{N}] = \emptyset
\end{array}
\right\},
\end{eqnarray}
where $\beta = \infty $ if no such $k$ exists. Note that $\beta = \infty $ implies $\overline{S}^{1}(\tau^{1}_{N}) = \overline{S}^{2}(\tau^{2}_{N})$ and 
\begin{equation*}
P^{w^{1}_{m/3},w^{2}_{m/3}}_{1,2}( \overline{S}^{1}(\tau^{1}_{N}) = \overline{S}^{2}(\tau^{2}_{N})) \le c 2^{-N}.
\end{equation*}
By definition of $\beta$, we see that $\beta = \frac{m}{3}$ implies that $R_{\frac{m}{3} , N}$ holds.
Therefore, if we let $J_{m,N}$ be the event
\begin{equation}\label{bmnotint}
J_{m,N}= \{ B^{1}[0, T^{1}_{N}] \cap \text{PATH}^{2}_{f} = \emptyset , B^{2}[0, T^{2}_{N}] \cap \text{PATH}^{1}_{f} = \emptyset , B^{1}[0, T^{1}_{N}] \cap B^{2}[0, T^{2}_{N}] = \emptyset \},
\end{equation}
then
\begin{equation}\label{sepcase}
P^{w^{1}_{m/3},w^{2}_{m/3}}_{1,2} ( J_{m,N} ) \le P^{w^{1}_{m/3},w^{2}_{m/3}}_{1,2} (R_{\frac{m}{3} , N} ) + \textstyle\sum\limits_{ k= \frac{m}{3} +1}^{N} P^{w^{1}_{m/3},w^{2}_{m/3}}_{1,2} ( J_{m,N} , \beta = k ) + c 2^{-N}. 
\end{equation}
Since 
\begin{equation*}
\min \{ \text{dist} ( w^{1}_{m/3} , \text{PATH}^{2}_{f} ) , \text{dist} ( w^{2}_{m/3} , \text{PATH}^{1}_{f} ) \} \ge 2^{\frac{0.99m}{3}-1}
\end{equation*}
on the event $F_{m}$, we see that $P^{w^{1}_{m/3},w^{2}_{m/3}}_{1,2} ( J_{m,N} ) > 0 $.
In this section, we will estimate $P^{w^{1}_{m/3},w^{2}_{m/3}}_{1,2} ( J_{m,N} , \beta = k ) $ for $ \frac{m}{3} < k \le N$ assuming that $F_{m}, G_{m}$ and $H_{m}$ hold.

\subsubsection{ Bounds for $  \frac{m}{3} < k \le \frac{21m}{60} $}\label{sub1}
Let $ \frac{m}{3} < k \le \frac{21m}{60} $. It is easy to see that $\beta = k$ implies that 
\begin{align*}
\overline{S}^{1}[\tau^{1}_{k-1}, \tau^{1}_{k}] \cap (\overline{S}^{2}[0, \tau^{2}_{k}&] \cup S^{2}[0,\tau^{2}_{ \frac{m}{3} }] \cup \gamma^{2} ) \neq \emptyset \\
&\text{or} \\
\overline{S}^{2}[\tau^{2}_{k-1}, \tau^{2}_{k}] \cap (\overline{S}^{1}[0, \tau^{1}_{k}&] \cup S^{1}[0,\tau^{1}_{ \frac{m}{3} }] \cup \gamma^{1} ) \neq \emptyset.
\end{align*}
We assume that the first event holds. (Similar arguments work for the second one.) 

\begin{lem}\label{step1}
There exist $\delta > 0$ and $c < \infty $ such that 
\begin{equation}\label{ineqstep1}
P^{w^{1}_{m/3},w^{2}_{m/3}}_{1,2} ( J_{m,N} , \overline{S}^{1}[\tau^{1}_{k-1}, \tau^{1}_{k}] \cap \overline{S}^{2}[0, \tau^{2}_{k}] \neq \emptyset ) \le c 2^{-(N-\frac{m}{3}) \xi } 2^{-\delta k}.
\end{equation}
\end{lem}

\begin{proof}
Let 
\begin{equation*}
Q^{i} = \{ \sup_{ 0 \le s \le T^{i}_{k+1} } |B^{i}(s) - \overline{S}^{i} (ds) | \ge 2^{\frac{31k}{60}} \},
\end{equation*}
and $Q= Q^{1} \cup Q^{2}$. Let $\Psi ^{1}(2^{k+1}), \Psi ^{2}(2^{k+1})$ be the events given in Proposition \ref{measurability} for $(B^{1}, \overline{S}^{1})$ and $(B^{2}, \overline{S}^{2})$, respectively, and let $\Psi = \Psi ^{1}(2^{k+1}) \cap \Psi ^{2}(2^{k+1})$. Then there exist $c < \infty$ and $\delta > 0$ such that 
\begin{equation}\label{Psihigh}
P ( J_{m,N} , \Psi ^{c} ) \le c 2^{-(N-k-2)\xi} \exp (- 2^{\delta k}).
\end{equation}
By Proposition \ref{measurability}, $Q$ and $\{B^{1}(t) : t \ge T^{1}_{k+2} \} \cup \{B^{2}(t) : t \ge T^{2}_{k+2} \}$ are conditionally independent given $B^{1}(T^{1}_{k+2}), B^{1}(T^{1}_{k+2})$ on the event $\Psi$. Hence by Proposition \ref{skorohod},
\begin{align}\label{q}
P( J_{m,N} , \Psi , Q ) &\le P ( B^{1}[T^{1}_{k+2}, T^{1}_{N}] \cap B^{2}[T^{2}_{k+2}, T^{2}_{N}] = \emptyset , \Psi , Q) \notag \\
& \le c 2^{-(N-k-2)\xi } \exp ( -2^{\delta k}).
\end{align}
Now we give an upper bound of
\begin{equation}\label{ess}
P(  J_{m,N}, \Psi , Q^{c} , \overline{S}^{1}[\tau^{1}_{k-1}, \tau^{1}_{k}] \cap \overline{S}^{2}[0, \tau^{2}_{k}] \neq \emptyset ).
\end{equation}
By the strong Markov property, this probability is bounded above by
\begin{equation*}
c 2^{-(N-k-2)\xi}P( J_{m , k+1} , \Psi , Q^{c} , \overline{S}^{1}[\tau^{1}_{k-1}, \tau^{1}_{k}] \cap \overline{S}^{2}[0, \tau^{2}_{k}] \neq \emptyset ).
\end{equation*}
Assume $Q^{c}$ holds. Then it is easy to see that 
\begin{equation*}
d T^{i}( 2^{k-1}-2^{\frac{31k}{60}} ) \le \tau^{i}_{k-1} < \tau^{i}_{k} \le d T^{i}( 2^{k}+2^{\frac{31k}{60}} ).
\end{equation*}
Hence on the event $Q^{c} \cap \{ \overline{S}^{1}[\tau^{1}_{k-1}, \tau^{1}_{k}] \cap \overline{S}^{2}[0, \tau^{2}_{k}] \neq \emptyset \}$, we see that there exist $s,t$ with
\begin{align*}
d T^{1}( 2^{k-1}-2^{\frac{31k}{60}} ) \le s \le d T^{1}( 2^{k}+2^{\frac{31k}{60}} ), \\
 0\le t \le d T^{2}( 2^{k}+2^{\frac{31k}{60}} )
\end{align*}
such that $\overline{S}^{1}(s)=\overline{S}^{2}(t)$. For such $s$ and $t$, we have
\begin{equation*}
|B^{1}(\frac{s}{d}) -B^{2}(\frac{t}{d})| \le 2^{\frac{31k}{60}+1}.
\end{equation*}
Namely, the following event holds,
\begin{equation}\label{bmclose}
D_{k}:=\{ \text{dist} \big( B^{1}[ T^{1}( 2^{k-1}-2^{\frac{31k}{60}} ) , T^{1}( 2^{k}+2^{\frac{31k}{60}} )] , B^{2}[0, T^{2}( 2^{k}+2^{\frac{31k}{60}} )] \big) \le 2^{\frac{31k}{60}+1} \}.
\end{equation}
Let 
\begin{equation*}
Z_{k} = \sup P^{z}(B[0, T_{k+1}] \cap B^{2}[0,T_{k+1}] = \emptyset ),
\end{equation*}
where the supremum is over all $z$ with $ z \in {\cal B} (2^{k} + 2^{\frac{31k}{60}})$ and 
\begin{equation*}
\text{dist} ( z, B^{2}[0, T^{2}(2^{k} + 2^{\frac{31k}{60}}) ) \le 2^{\frac{31k}{60}+1}.
\end{equation*}
We let $H_{k}$ be the event $\{ Z_{k} \le 2^{-\delta k} \}$. By Proposition \ref{hittable for bm}, there exists $\delta > 0$ such that
\begin{equation*}
P(H_{k}) \le 2^{-6k}.
\end{equation*}
Therefore, we have only to estimate 
\begin{equation*}
P( J_{m , k+1} , \Psi , D_{k} , H_{k}^{c} ).
\end{equation*}
On the event $J_{m , k+1} \cap D_{k} \cap H_{k}^{c}$, $B^{1}[ T^{1}( 2^{k-1}-2^{\frac{31k}{60}} ) , T^{1}_{k+1}]$ does not intersect $B^{2}[0, T^{2}_{k+1}]$ nevertheless $B^{1}$ gets close to $B^{2}[0, T^{2}_{k+1}]$ which is a hittable set. By the strong Markov property,
\begin{equation*}
P( J_{m , k+1} , \Psi , D_{k} , H_{k}^{c} ) \le c 2^{-\delta k} 2^{-(k-\frac{m}{3}) \xi },
\end{equation*}
and this finishes the proof.

\end{proof}

\begin{lem}\label{step2}
There exist $\delta > 0$ and $c < \infty $ such that 
\begin{equation}\label{ineqstep2}
P^{w^{1}_{m/3},w^{2}_{m/3}}_{1,2} ( J_{m,N} , \overline{S}^{1}[\tau^{1}_{k-1}, \tau^{1}_{k}] \cap (S^{2}[0,\tau^{2}_{ \frac{m}{3} }] \cup \gamma^{2} ) \neq \emptyset ) \le c 2^{-(N-\frac{m}{3}) \xi } \exp (- 2^{\delta k}).
\end{equation}
\end{lem}

\begin{proof}
Recall $\Psi$ and $Q$ are the events given in the proof of Lemma \ref{step1}. By \eqref{Psihigh} and \eqref{q}, it suffices to estimate 
\begin{equation*}
P^{w^{1}_{m/3},w^{2}_{m/3}}_{1,2} ( J_{m,N} , \overline{S}^{1}[\tau^{1}_{k-1}, \tau^{1}_{k}] \cap (S^{2}[0,\tau^{2}_{ \frac{m}{3} }] \cup \gamma^{2} ) \neq \emptyset , \Psi , Q^{c} ) .
\end{equation*}
By the strong Markov property, this probability is bounded above by
\begin{equation*}
c 2^{-(N-k) \xi } P^{w^{1}_{m/3},w^{2}_{m/3}}_{1,2} \big( J_{m , k+1} , \Psi , Q^{c} , \overline{S}^{1}[\tau^{1}_{k-1}, \tau^{1}_{k}] \cap (S^{2}[0,\tau^{2}_{ \frac{m}{3} }] \cup \gamma^{2} ) \neq \emptyset \big).
\end{equation*}
On the event $Q^{c} \cap \{ \overline{S}^{1}[\tau^{1}_{k-1}, \tau^{1}_{k}] \cap (S^{2}[0,\tau^{2}_{ \frac{m}{3} }] \cup \gamma^{2} ) \neq \emptyset \}$, it is easy to see that there exists $t$ with 
\begin{equation*}
T^{1}( 2^{k-1}-2^{\frac{31k}{60}} ) \le t \le T^{1}( 2^{k}+2^{\frac{31k}{60}})
\end{equation*}
such that
\begin{equation*}
\text{dist}\big( B^{1}(t) , (S^{2}[0,\tau^{2}_{ \frac{m}{3} }] \cup \gamma^{2} ) \big) \le 2^{\frac{31k}{60}}.
\end{equation*}
Since $k \le \frac{21m}{60}$, we have $ \frac{31k}{60} \le \frac{651m}{3600} < \frac{11m}{60}$. Therefore, 
\begin{equation*}
B^{1}[0,T^{1}_{k+1}] \cap \text{PATH} ^{2}_{f} \neq \emptyset ,
\end{equation*}
and the lemma is finished.

\end{proof}

\subsubsection{ Bounds for $  \frac{21m}{60} < k \le N-3 $}\label{sub2}
From now we assume that $  \frac{21m}{60} < k \le N-3 $. The similar argument in the proof of Lemma \ref{step1} gives the following lemma, so we omit the proof.
\begin{lem}\label{step3}
There exist $\delta > 0$ and $c < \infty $ such that 
\begin{equation}\label{ineqstep3}
P^{w^{1}_{m/3},w^{2}_{m/3}}_{1,2} ( J_{m,N} , \overline{S}^{1}[\tau^{1}_{k-1}, \tau^{1}_{k}] \cap \overline{S}^{2}[0, \tau^{2}_{k}] \neq \emptyset ) \le c 2^{-(N-\frac{m}{3}) \xi } 2^{-\delta k}.
\end{equation}
\end{lem}

Let $\Psi$ and $Q$ be the events defined in the proof of Lemma \ref{step1}. By \eqref{Psihigh} and \eqref{q}, in order to prove
\begin{equation*}
P^{w^{1}_{m/3},w^{2}_{m/3}}_{1,2} ( J_{m,N}, \beta =k) \le  c 2^{-(N-\frac{m}{3}) \xi } 2^{-\delta k},
\end{equation*}
for $  \frac{21m}{60} < k \le N-3 $, it is enough to show the following lemma.

\begin{lem}\label{4step}
There exist $\delta > 0$ and $c < \infty$ such that
\begin{equation}\label{ineqstep4}
P^{w^{1}_{m/3},w^{2}_{m/3}}_{1,2} ( J_{m,N} , \overline{S}^{1}[\tau^{1}_{k-1}, \tau^{1}_{k}] \cap (S^{2}[0,\tau^{2}_{ \frac{m}{3} }] \cup \gamma^{2} ) \neq \emptyset , \Psi , Q^{c} )  \le c 2^{-(N-\frac{m}{3}) \xi } 2^{-\delta k}.
\end{equation}
\end{lem}
\begin{proof}
By the strong Markov property, the right hand side of \eqref{ineqstep4} is bounded above by 
\begin{equation*}
c 2^{-(N-k) \xi } P^{w^{1}_{m/3},w^{2}_{m/3}}_{1,2} \big( J_{m , k+1} , \Psi , Q^{c} , \overline{S}^{1}[\tau^{1}_{k-1}, \tau^{1}_{k}] \cap (S^{2}[0,\tau^{2}_{ \frac{m}{3} }] \cup \gamma^{2} ) \neq \emptyset \big).
\end{equation*}
Assume $d=3$ and $ \frac{21m}{60} < k \le \frac{20m}{31}$ so that $ \frac{31k}{60} \le \frac{m}{3}$. If $\overline{S}^{1}[\tau^{1}_{k-1}, \tau^{1}_{k}] \cap (S^{2}[0,\tau^{2}_{ \frac{m}{3} }] \cup \gamma^{2} ) \neq \emptyset$, then $\overline{S}^{1}[\tau^{1}_{k-1}, \tau^{1}_{k}] \cap {\cal B} ( 2^{\frac{m}{3}} ) \neq \emptyset$. On the other hand, on the event $Q^{c}$, we have
\begin{equation*}
3 T^{1}( 2^{k-1}-2^{\frac{31k}{60}} ) \le \tau^{1}_{k-1} \le \tau^{1}_{k} \le 3T^{1}( 2^{k}+2^{\frac{31k}{60}}).
\end{equation*}
Since $\frac{31k}{60} \le \frac{m}{3}$, we have
\begin{equation*}
B^{1}[T^{1}( 2^{k-1}-2^{\frac{31k}{60}} ) , T^{1}( 2^{k}+2^{\frac{31k}{60}} )] \cap {\cal B} ( 2^{\frac{m}{3}+1} )\neq \emptyset.
\end{equation*}
For $k > \frac{21m}{60}$, a standard estimate shows that
\begin{equation*}
P_{1}(B^{1}[T^{1}( 2^{k-1}-2^{\frac{31k}{60}} ) , T^{1}( 2^{k}+2^{\frac{31k}{60}} )] \cap {\cal B} ( 2^{\frac{m}{3}+1} )\neq \emptyset) \le c 2^{-(k-\frac{m}{3})}.
\end{equation*}
Using the strong Markov property at $T^{1}( 2^{k-1}-2^{\frac{31k}{60}} )$ first, and then estimating $P( J_{m , k-2})$, we have
\begin{equation*}
P^{w^{1}_{m/3},w^{2}_{m/3}}_{1,2} \big( J_{m , k+1} , \Psi , Q^{c} , \overline{S}^{1}[\tau^{1}_{k-1}, \tau^{1}_{k}] \cap (S^{2}[0,\tau^{2}_{ \frac{m}{3} }] \cup \gamma^{2} ) \neq \emptyset \big) \le c  2^{-(k-\frac{m}{3})} 2^{-(k-\frac{m}{3})\xi }.
\end{equation*}
Therefore, the proof for $d=3$ and $ \frac{21m}{60} < k \le \frac{20m}{31}$ is finished.

Next we assume $d=3$ and $ \frac{20m}{31} < k \le N-3$. In this case, if $\overline{S}^{1}[\tau^{1}_{k-1}, \tau^{1}_{k}] \cap {\cal B} ( 2^{\frac{m}{3}} ) \neq \emptyset$ and $Q^{c}$ hold, then 
\begin{equation}\label{return}
B^{1}[T^{1}( 2^{k-1}-2^{\frac{31k}{60}} ) , T^{1}( 2^{k}+2^{\frac{31k}{60}} )] \cap {\cal B} ( 2^{\frac{31k}{60}+1} )\neq \emptyset.
\end{equation}
Since this event occur with probability at most $c 2^{-\frac{k}{3}}$, the lemma is proved for $d=3$.

Assume $d=2$. In this case, the probability of the event \eqref{return} is bounded below by $1/k$, so we need to change the proof. Assume $ \frac{21m}{60} < k \le \frac{20m}{31}$. (For the otherwise, the proof is almost same in this case. So we only consider this case.) Let
\begin{equation*}
\eta = \inf \{ t \ge T^{1}( 2^{k-1}-2^{\frac{31k}{60}} ) : B^{1}(t) \in {\cal B} ( 2^{\frac{m}{3}+1} ) \}.
\end{equation*}
 We already showed that if $\overline{S}^{1}[\tau^{1}_{k-1}, \tau^{1}_{k}] \cap (S^{2}[0,\tau^{2}_{ \frac{m}{3} }] \cup \gamma^{2} ) \neq \emptyset$ and $Q^{c}$ hold, then $ \eta \le T^{1}( 2^{k}+2^{\frac{31k}{60}} )$. By the Proposition \ref{beurling estimate}, we see that
\begin{align*}
&P^{w^{1}_{m/3},w^{2}_{m/3}}_{1,2} \big(  J_{m , k+1} , \eta \le T^{1}( 2^{k}+2^{\frac{31k}{60}} ) \big) \notag \\
&\le E^{w^{2}_{m/3}}_{2} \big( E^{w^{1}_{m/3}}_{1} \big( {\bf 1} \{ J_{m, k-2} , \eta \le T^{1}( 2^{k}+2^{\frac{31k}{60}})  , B^{1}[\eta , T^{1}_{k+1}] \cap B^{2}[0, T^{2}_{k+1}]= \emptyset \} \big) \big) \notag \\
&\le P^{w^{1}_{m/3},w^{2}_{m/3}}_{1,2} (  J_{m , k-2} ) c 2^{-\frac{k}{5}} \notag \\
&\le c 2^{-(k-\frac{m}{3})\xi } 2^{ -\frac{k}{5}},
\end{align*}
and the lemma is proved for all cases.

\end{proof}

\subsubsection{ Bounds for $ N-2 \le k \le N $}\label{sub3}
Finally, we give estimates for $ N-2 \le k \le N $. Since a proof is similar for each case, we only consider for $k=N$. By definition of $\beta$ in \eqref{firstintrw}, we see that $\beta = N$ implies the event
\begin{equation}\label{k=N}
\bigcup _{i=1,2} \{ \overline{S}^{i}[\tau ^{i}_{N-1}, \tau^{i}_{N}] \cap ( \overline{S}^{3-i}[0, \tau ^{3-i}_{N}]  \cup S^{3-i}[0,\tau^{3-i}_{ \frac{m}{3} }] \cup \gamma^{3-i} ) \neq \emptyset \}.
\end{equation}
We will only give bounds on the probability of the event for $i=1$ in \eqref{k=N}. First we show the following lemma.

\begin{lem}\label{5step}
There exist $ \delta > 0$ and $ c < \infty $ such that
\begin{equation}\label{ineqstep5}
P^{w^{1}_{m/3},w^{2}_{m/3}}_{1,2} ( J_{m,N} , \overline{S}^{1}[\tau^{1}_{N-1}, \tau^{1}_{N}] \cap (S^{2}[0,\tau^{2}_{ \frac{m}{3} }] \cup \gamma^{2} ) \neq \emptyset ) \le c 2^{-(N-\frac{m}{3}) \xi } 2^{-\delta N}.
\end{equation}
\end{lem}

\begin{proof}
Let
\begin{equation*}
Q^{i}= \{ \sup _{0 \le t \le T^{i}_{N+1}} |B^{i}(t)-\overline{S}^{i}(dt) | \ge 2^{\frac{31N}{60}} \}
\end{equation*}
and $Q=Q^{1} \cup Q^{2}$. By Proposition \ref{skorohod},
\begin{equation}\label{qislaw}
P^{w^{1}_{m/3},w^{2}_{m/3}}_{1,2}(Q) \le c \exp (-(2^{\delta N})).
\end{equation}
Assume $Q^{c}$ holds. Then $dT^{1}( 2^{N-1}-2^{\frac{31N}{60}} ) \le \tau^{1}_{N-1} $. Therefore, if $\overline{S}^{1}[\tau^{1}_{N-1}, \tau^{1}_{N}] \cap (S^{2}[0,\tau^{2}_{ \frac{m}{3} }] \cup \gamma^{2} ) \neq \emptyset$, we have
\begin{equation}\label{returning}
B^{1}[T^{1}_{N-2}, \infty ) \cap {\cal B}(2^{ \frac{31N}{60}+1} ) \neq \emptyset.
\end{equation}
For $d=3$, the probability of the event \eqref{returning} is bounded above by $c 2^{-\frac{N}{3}}$. Therefore, by the strong Markov property,
\begin{align*}
&P^{w^{1}_{m/3},w^{2}_{m/3}}_{1,2} ( J_{m,N} , \overline{S}^{1}[\tau^{1}_{N-1}, \tau^{1}_{N}] \cap (S^{2}[0,\tau^{2}_{ \frac{m}{3} }] \cup \gamma^{2} ) \neq \emptyset )  \\
& \le P^{w^{1}_{m/3},w^{2}_{m/3}}_{1,2}( J_{m, N-2} , Q^{c} , B^{1}[T^{1}_{N-2}, \infty ) \cap {\cal B}(2^{ \frac{31N}{60}+1} ) \neq \emptyset ) + c \exp (-(2^{\delta N})) \\
& \le c 2^{-\frac{N}{3}}  2^{-(N-\frac{m}{3}) \xi },
\end{align*}
for $d=3$.

Next we consider the two dimensional case. Assume $\overline{S}^{1}[\tau^{1}(2^{N}- 2^{\frac{31N}{60}}), \tau^{1}_{N}] \cap (S^{2}[0,\tau^{2}_{ \frac{m}{3} }] \cup \gamma^{2} ) \neq \emptyset$ and $Q^{c}$ holds. This implies that $\overline{S}^{1}[\tau^{1}(2^{N}- 2^{\frac{31N}{60}}), \tau^{1}_{N}] \cap {\cal B}(2^{ \frac{31N}{60}} ) \neq \emptyset$. On the event $Q^{c}$, we have
\begin{equation*}
2T^{1}(2^{N}- 2^{\frac{31N}{60}+1}) \le \tau^{1}(2^{N}- 2^{\frac{31N}{60}}) \le \tau^{1}_{N} \le 2T^{1}(2^{N}+ 2^{\frac{31N}{60}}).
\end{equation*}
Therefore,
\begin{equation}\label{bmenter}
B^{1}[T^{1}(2^{N}- 2^{\frac{31N}{60}+1}), T^{1}(2^{N}+ 2^{\frac{31N}{60}})] \cap {\cal B} (2^{ \frac{31N}{60} +1} ) \neq \emptyset.
\end{equation}
Using Proposition \ref{beurling estimate}, the probability of the event \eqref{bmenter} is bounded above by $c 2^{-\frac{N}{3}}$. Hence by the strong Markov property, 
\begin{equation*}
P^{w^{1}_{m/3},w^{2}_{m/3}}_{1,2} ( J_{m,N} , \overline{S}^{1}[\tau^{1}(2^{N}- 2^{\frac{31N}{60}}), \tau^{1}_{N}] \cap (S^{2}[0,\tau^{2}_{ \frac{m}{3} }] \cup \gamma^{2} ) \neq \emptyset ) \le c 2^{-\frac{N}{3}} 2^{-(N-\frac{m}{3}) \xi }.
\end{equation*}
Assume $\overline{S}^{1}[ \tau^{1}_{N-1} , \tau^{1}(2^{N}- 2^{\frac{31N}{60}})] \cap (S^{2}[0,\tau^{2}_{ \frac{m}{3} }] \cup \gamma^{2} ) \neq \emptyset$ and $Q^{c}$ holds. This implies that 
\begin{equation}\label{in}
B^{1}[T^{1}(2^{N-1}- 2^{\frac{31N}{60}}), T^{1}_{N}] \cap {\cal B} (2^{ \frac{31N}{60} +1} ) \neq \emptyset.
\end{equation}
So let 
\begin{equation*}
\rho = \inf \{ t \ge T^{1}(2^{N-1}- 2^{\frac{31N}{60}}) : B^{1}(t) \in{\cal B} (2^{ \frac{31N}{60} +1} ) \}.
\end{equation*}
Again by using Proposition \ref{beurling estimate},
\begin{align*}
&P^{w^{1}_{m/3},w^{2}_{m/3}}_{1,2} ( J_{m,N} , \overline{S}^{1}[ \tau^{1}_{N-1} , \tau^{1}(2^{N}- 2^{\frac{31N}{60}})] \cap (S^{2}[0,\tau^{2}_{ \frac{m}{3} }] \cup \gamma^{2} ) \neq \emptyset , Q^{c} ) \\
& \le P^{w^{1}_{m/3},w^{2}_{m/3}}_{1,2} ( J_{m,N-2} , \rho \in [T^{1}(2^{N-1}- 2^{\frac{31N}{60}}), T^{1}_{N} ] , B^{1}[\rho , T^{1}_{N}] \cap B^{2}[0,T^{2}_{N}] = \emptyset ) \\
& \le  P^{w^{1}_{m/3},w^{2}_{m/3}}_{1,2} ( J_{m,N-2}) 2^{-\frac{N}{3}} \\
&\le c 2^{-(N-\frac{m}{3}) \xi } 2^{-\frac{N}{3}},
\end{align*}
and the lemma is proved.

\end{proof}

To estimate the probability of \eqref{k=N}, we have only to show the following lemma.

\begin{lem}\label{6step}
There exist $ \delta > 0$ and $ c < \infty $ such that
\begin{equation}\label{ineqstep6}
P^{w^{1}_{m/3},w^{2}_{m/3}}_{1,2} ( J_{m,N} , \overline{S}^{1}[\tau^{1}_{N-1}, \tau^{1}_{N}] \cap \overline{S}^{2}[0, \tau ^{2}_{N}] \neq \emptyset ) \le c 2^{-(N-\frac{m}{3}) \xi } 2^{-\delta N}.
\end{equation}
\end{lem}

Before we start to prove this lemma, we need to prepare several lemmas.

\begin{lem}\label{6stepsub1}
There exist $ \delta > 0$ and $ c < \infty $ such that
\begin{equation}\label{ineqstep6sub1}
P^{w^{1}_{m/3},w^{2}_{m/3}}_{1,2} ( J_{m,N} , \overline{S}^{1}[\tau^{1}_{N-1}, \tau^{1}_{N}] \cap \overline{S}^{2}[\tau^{2}(2^{N}-2^{\frac{2N}{3}}), \tau ^{2}_{N}] \neq \emptyset ) \le c 2^{-(N-\frac{m}{3}) \xi } 2^{-\delta N}.
\end{equation}
\end{lem}

\begin{proof}
Let $Q$ be the event defined in the proof of Lemma \ref{5step}. Let 
\begin{equation*}
\sigma = \inf \{ k \ge \tau^{2}(2^{N}-2^{\frac{2N}{3}}) : \overline{S}^{2}_{k} \in \partial {\cal B} ( \overline{S}^{2} ( \tau^{2}(2^{N}-2^{\frac{2N}{3}})) , 2^{\frac{3N}{4}}) \}.
\end{equation*}
If $Q^{c}$ holds and $\sigma < \tau^{2}_{N}$, then
\begin{equation}\label{nointboundary}
\overline{\sigma}:= \inf \{ t \ge T^{2}( 2^{N}-2^{\frac{2N}{3}} + 2^{\frac{31N}{60}}) : B^{2}(t) \in \partial {\cal B} ( B^{2} ( T^{2}( 2^{N}-2^{\frac{2N}{3}} + 2^{\frac{31N}{60}}) , 2^{\frac{3N}{4}} - 2^{\frac{31N}{60}}) \} \le T^{2}( 2^{N} + 2^{\frac{31N}{60}} ).
\end{equation}
It is easy to see that the probability of \eqref{nointboundary} is bounded above by $c 2^{-\delta N}$ for some $c < \infty$ and $\delta >0$. Hence by the strong Markov property, 
\begin{equation*}
P^{w^{1}_{m/3},w^{2}_{m/3}}_{1,2} ( J_{m,N} , \overline{S}^{1}[\tau^{1}_{N-1}, \tau^{1}_{N}] \cap \overline{S}^{2}[\tau^{2}(2^{N}-2^{\frac{2N}{3}}), \tau ^{2}_{N}] \neq \emptyset , \sigma < \tau^{2}_{N} ) \le c 2^{-\delta N} 2^{-(N-\frac{m}{3}) \xi}.
\end{equation*}
Now assume $\sigma < \tau^{2}_{N}$. Then $\overline{S}^{2}[\tau^{2}(2^{N}-2^{\frac{2N}{3}}), \tau ^{2}_{N}] \subset {\cal B} ( \overline{S}^{2} ( \tau^{2}(2^{N}-2^{\frac{2N}{3}})) , 2^{\frac{3N}{4}})$. Therefore $\overline{S}^{1}[\tau^{1}_{N-1}, \tau^{1}_{N}] \cap \overline{S}^{2}[\tau^{2}(2^{N}-2^{\frac{2N}{3}}), \tau ^{2}_{N}] \neq \emptyset$ implies that 
\begin{equation}\label{rwballhit}
\overline{S}^{1}[\tau^{1}_{N-1}, \tau^{1}_{N}] \cap {\cal B} ( \overline{S}^{2} ( \tau^{2}(2^{N}-2^{\frac{2N}{3}})) , 2^{\frac{3N}{4}}) \neq \emptyset.
\end{equation}
If $Q^{c}$ and \eqref{rwballhit} hold, we see that 
\begin{equation}\label{bmballhit}
B^{1}[T^{1}(2^{N-1}-2^{\frac{31N}{60}} ) , T^{1}( 2^{N}+2^{\frac{31N}{60}} )] \cap {\cal B} ( B^{2}( T^{2}(2^{N}-2^{\frac{2N}{3}})) , 2^{\frac{3N}{4}+1}) \neq \emptyset.
\end{equation}
For any $ x \in \partial {\cal B} ( 2^{N}-2^{\frac{2N}{3}})$, we have
\begin{equation*}
P_{1} ( B^{1} [T^{1}(2^{N-1}-2^{\frac{31N}{60}} ) , T^{1}( 2^{N}+2^{\frac{31N}{60}} )] \cap {\cal B} ( x, 2^{\frac{3N}{4}+1}) \neq \emptyset ) \le c 2^{-\frac{N}{4}}.
\end{equation*}
By the strong Markov property,
\begin{equation*}
P^{w^{1}_{m/3},w^{2}_{m/3}}_{1,2} ( J_{m,N} , \overline{S}^{1}[\tau^{1}_{N-1}, \tau^{1}_{N}] \cap \overline{S}^{2}[\tau^{2}(2^{N}-2^{\frac{2N}{3}}), \tau ^{2}_{N}] \neq \emptyset , \sigma \ge \tau^{2}_{N} ) \le c 2^{-\delta N} 2^{-(N-\frac{m}{3}) \xi},
\end{equation*}
and hence prove the lemma.

\end{proof}

\begin{rem}\label{rem1}
Similar arguments in the proof of Lemma \ref{6stepsub1} give that 
\begin{equation}\label{omit}
P^{w^{1}_{m/3},w^{2}_{m/3}}_{1,2} ( J_{m,N} , \overline{S}^{1}[\tau^{1}(2^{N}-2^{\frac{2N}{3}}), \tau^{1}_{N}] \cap \overline{S}^{2}[0, \tau ^{2}_{N}] \neq \emptyset ) \le c 2^{-(N-\frac{m}{3}) \xi } 2^{-\delta N}.
\end{equation}
\end{rem}

By Lemma \ref{6stepsub1} and Remark \ref{rem1}, we have only to show the following lemma to prove Lemma \ref{6step}.

\begin{lem}\label{6stepsub2}
There exist $ \delta > 0$ and $ c < \infty $ such that
\begin{equation}\label{ineqstep6sub2}
P^{w^{1}_{m/3},w^{2}_{m/3}}_{1,2} ( J_{m,N} , \overline{S}^{1}[\tau^{1}_{N-1}, \tau^{1}(2^{N}-2^{\frac{2N}{3}})] \cap \overline{S}^{2}[0 , \tau^{2}(2^{N}-2^{\frac{2N}{3}})] \neq \emptyset ) \le c 2^{-(N-\frac{m}{3}) \xi } 2^{-\delta N}.
\end{equation}
\end{lem}

\begin{proof}
Let $Q$ be the event defined in the proof of Lemma \ref{5step}. If 
\begin{equation*}
\overline{S}^{1}[\tau^{1}_{N-1}, \tau^{1}(2^{N}-2^{\frac{2N}{3}})] \cap \overline{S}^{2}[0 , \tau^{2}(2^{N}-2^{\frac{2N}{3}})] \neq \emptyset
\end{equation*}
and $Q^{c}$ holds, then we have
\begin{equation}\label{intersection}
\text{dist} ( B^{1}[ T^{1}(2^{N-1}-2^{ \frac{31N}{60} }) , T^{1}(2^{N} -2^{\frac{2N}{3}} +2^{ \frac{31N}{60} }) ] , B^{2} [ 0, T^{2}( 2^{N}-2^{\frac{2N}{3}} + 2^{ \frac{31N}{60} } ) ] ) \le 2^{ \frac{31N}{60}  + 1}.
\end{equation}
Let 
\begin{equation*}
Z= \sup P^{z}( B[0, T_{N}] \cap B^{2}[0, T^{2}_{N}] = \emptyset ),
\end{equation*}
where the supremum is over all $z$ with $ z \in {\cal B} ( 2^{N} -2^{\frac{2N}{3}} +2^{ \frac{31N}{60} } ) $ and 
\begin{equation*}
\text{dist} ( z, B^{2} [ 0, T^{2}( 2^{N}-2^{\frac{2N}{3}} + 2^{ \frac{31N}{60} } ) ] ) \le 2^{ \frac{31N}{60}  + 1}.
\end{equation*}
Then by Proposition \ref{hittable for bm}, 
\begin{equation*}
P^{w^{2}_{m/3}}_{2} ( Z \ge 2^{-\delta N} ) \le c2^{-6N},
\end{equation*}
for some $\delta > 0$ and $c < \infty$. Therefore, 
\begin{equation*}
P^{w^{1}_{m/3},w^{2}_{m/3}}_{1,2} ( J_{m,N} , \overline{S}^{1}[\tau^{1}_{N-1}, \tau^{1}(2^{N}-2^{\frac{2N}{3}})] \cap \overline{S}^{2}[0 , \tau^{2}(2^{N}-2^{\frac{2N}{3}})] \neq \emptyset , Q^{c} ) 
\end{equation*}
is bounded above by
\begin{align*}\label{3dimesti}
P^{w^{1}_{m/3},w^{2}_{m/3}}_{1,2} ( J_{m,N} , &\text{dist} ( B^{1}[ T^{1}(2^{N-1}-2^{ \frac{31N}{60} }) , T^{1}(2^{N} -2^{\frac{2N}{3}} +2^{ \frac{31N}{60} }) ] , B^{2} [ 0, T^{2}( 2^{N}-2^{\frac{2N}{3}} + 2^{ \frac{31N}{60} } ) ] )  \\
& \le 2^{ \frac{31N}{60}  + 1} , Z \le 2^{-\delta N} ) + c2^{-6N}.
\end{align*}
Using the strong Markov property for $B^{1}$, we see that this probability is bounded above by 
\begin{equation*}
2^{-\delta N}  P^{w^{1}_{m/3},w^{2}_{m/3}}_{1,2} ( J_{m, N-2}),
\end{equation*}
and hence the proof is finished.

\end{proof}

\subsubsection{Conclusion \ Lower Bound}
Combining estimates obtained in subsections \ref{sub1}, \ref{sub2} and \ref{sub3} with \eqref{sepcase}, we have the following proposition.
\begin{prop}\label{lowerbound}
There exist $\delta > 0$ and $c < \infty$ such that 
\begin{equation}\label{lowerboundesti}
P^{w^{1}_{m/3},w^{2}_{m/3}}_{1,2} ( J_{m,N}) \le P^{w^{1}_{m/3},w^{2}_{m/3}}_{1,2} ( R_{\frac{m}{3}, N} )  + c 2^{-(N-\frac{m}{3}) \xi } 2^{-\delta m },
\end{equation}
on the event $F_{m} \cap G_{m} \cap H_{m}$.
\end{prop}

\subsubsection{Upper bound}

From this subsection, we will give an upper bound of $P^{w^{1}_{m/3},w^{2}_{m/3}}_{1,2}(R_{\frac{m}{3} , N})$ by using $P^{w^{1}_{m/3},w^{2}_{m/3}}_{1,2} (J_{m,N})$ on the event $F_{m} \cap G_{m} \cap H_{m}$. For this purpose, define
\begin{eqnarray}\label{firstintbm}
\beta ^{\sharp} = \inf \left\{ \frac{m}{3} \le k \le N :
\begin{array}{ll}
B^{1}[T ^{1}_{k}, T^{1}_{N}] \cap ( B^{2}[0, T ^{2}_{k}]  \cup \text{PATH}^{2}_{f} )= \emptyset    \\
B^{2}[T ^{2}_{k}, T^{2}_{N}] \cap ( B^{1}[0, T ^{1}_{k}]  \cup \text{PATH}^{1}_{f} )= \emptyset    \\
B^{1}[T ^{1}_{k}, T^{1}_{N}] \cap B^{2}[T ^{2}_{k}, T^{2}_{N}] = \emptyset
\end{array}
\right\}.
\end{eqnarray}
Note that $\beta ^{\sharp} \le N$ almost surely and $\beta ^{\sharp}= \frac{m}{3}$ implies $J_{m,N}$ holds. Therefore,
\begin{align}\label{upperbound}
P^{w^{1}_{m/3},w^{2}_{m/3}}_{1,2}(R_{\frac{m}{3} , N}) & = P^{w^{1}_{m/3},w^{2}_{m/3}}_{1,2}(R_{\frac{m}{3} , N} , \frac{m}{3} \le \beta ^{\sharp} \le N ) \notag \\
&=P^{w^{1}_{m/3},w^{2}_{m/3}}_{1,2}(R_{\frac{m}{3} , N} , J_{m,N} ) + \textstyle\sum\limits_{ k= \frac{m}{3} +1}^{N} P^{w^{1}_{m/3},w^{2}_{m/3}}_{1,2}(R_{\frac{m}{3} , N} , \beta ^{\sharp} = k ) \notag \\
&\le P^{w^{1}_{m/3},w^{2}_{m/3}}_{1,2}( J_{m,N}) + \textstyle\sum\limits_{ k= \frac{m}{3} +1}^{N} P^{w^{1}_{m/3},w^{2}_{m/3}}_{1,2}(R_{\frac{m}{3} , N} , \beta ^{\sharp} = k ).
\end{align}
We will give bounds for the second term in the right hand side of \eqref{upperbound}. For $ \frac{m}{3} +1 \le k \le N$, $\beta ^{\sharp} = k$ implies that
\begin{equation}\label{rwnohitbmhit}
\bigcup _{i=1,2} \{ B^{1}[T ^{1}_{k}, T^{1}_{N}] \cap B^{2}[T ^{2}_{k}, T^{2}_{N}] = \emptyset \} \cap \{ B^{i}[T^{i}_{k-1}, T^{i}_{k}] \cap ( B^{3-i}[0,T^{3-i}_{k}] \cup \text{PATH}^{3-i}_{f} ) \neq \emptyset \}.
\end{equation}
We only consider for $i=1$ in \eqref{rwnohitbmhit}.

\subsubsection{Bounds for $ \frac{21m}{60} \le k \le N-3 $}\label{sub1-1}
\begin{lem}\label{sub1-1}
There exist $\delta > 0$ and $ c < \infty $ such that 
\begin{equation}\label{ineqsub1-1}
P^{w^{1}_{m/3},w^{2}_{m/3}}_{1,2}( B^{1}[T ^{1}_{k}, T^{1}_{N}] \cap B^{2}[T ^{2}_{k}, T^{2}_{N}] = \emptyset , B^{1}[T^{1}_{k-1}, T^{1}_{k}] \cap B^{2}[0,T^{2}_{k}] \neq \emptyset , R_{\frac{m}{3} , N} ) \le c 2^{ -(N- \frac{m}{3}) \xi } 2^{-\delta k} .
\end{equation}
\end{lem}

\begin{proof}
Since the idea is quite similar as in the proof of Lemma \ref{step1}, we will just sketch the proof. By the strong Markov property, the probability in the left hand side of \eqref{ineqsub1-1} can be bounded above by
\begin{equation*}
c 2^{-(N-k) \xi } P^{w^{1}_{m/3},w^{2}_{m/3}}_{1,2}(  B^{1}[T^{1}_{k-1}, T^{1}_{k}] \cap B^{2}[0,T^{2}_{k}] \neq \emptyset , R_{\frac{m}{3} , k+1} ).
\end{equation*}
By Proposition \ref{skorohod}, if $B^{1}[T^{1}_{k-1}, T^{1}_{k}] \cap B^{2}[0,T^{2}_{k}] \neq \emptyset$, then $\overline{S}^{1}$ gets close to $\overline{S}^{2}[0 , \tau ^{2}_{k}]$ during $[\tau ^{1}_{k-1} , \tau^{1}_{k}]$ with probability at least $1- c\exp ( -2^{\delta k} )$, for some $\delta > 0$ and $c< \infty$. By Proposition \ref{hittable for srw}, once $\overline{S}^{1}$ gets close to $\overline{S}^{2}[0 , \tau ^{2}_{k}]$ during $[\tau ^{1}_{k-1} , \tau^{1}_{k}]$, then $\overline{S}^{1}$ intersects $\overline{S}^{2}[0 , \tau ^{2}_{k+1}]$ until $\tau^{1}_{k+1}$ with probability at least $1-2^{-\delta k}$. Hence by using the strong Markov property, the lemma can be proved. 

\end{proof}

\begin{lem}\label{sub1-2}
There exist $\delta > 0$ and $ c < \infty $ such that 
\begin{equation}\label{ineqsub1-2}
P^{w^{1}_{m/3},w^{2}_{m/3}}_{1,2}( B^{1}[T ^{1}_{k}, T^{1}_{N}] \cap B^{2}[T ^{2}_{k}, T^{2}_{N}] = \emptyset , B^{1}[T^{1}_{k-1}, T^{1}_{k}] \cap \text{PATH}^{2}_{f} \neq \emptyset , R_{\frac{m}{3} , N} ) \le c 2^{ -(N- \frac{m}{3}) \xi } 2^{-\delta k} .
\end{equation}
\end{lem}

\begin{proof}
Similar ideas as in the proof of Lemma \ref{4step} works here. So we just state the idea of the proof. 

First let $d=3$. The probability that $B^{1}[T ^{1}_{k}, T^{1}_{N}]$ does not intersect $B^{2}[T ^{2}_{k}, T^{2}_{N}]$ is bounded above by $c 2^{-(N-k) \xi}$. Assume $B^{1}[T^{1}_{k-1}, T^{1}_{k}] \cap \text{PATH}^{2}_{f} \neq \emptyset$, then $B^{1}$ enters in ${\cal B} (2^{\frac{m}{3}})$ during $[ T^{1}_{k-1}, T^{1}_{k}]$. The probability that such an entrance occurs is at most $c 2^{-(k-\frac{m}{3})}$. Finally, using $P^{w^{1}_{m/3},w^{2}_{m/3}}_{1,2} ( R_{\frac{m}{3} , k-2} ) \le c 2^{-(k-\frac{m}{3})\xi }$ and the strong Markov property, the lemma is finished for $d=3$.

Next let $d=2$. In this case, if $B^{1}$ enters in ${\cal B} (2^{\frac{m}{3}})$ during $[ T^{1}_{k-1}, T^{1}_{k}]$, then $\overline{S}^{1}$ enters ${\cal B} (2^{(\frac{m}{3} \vee \frac{31N}{60})+1 })$ during $[\tau ^{1}_{k-1} , \tau^{1}_{k}]$ with probability at least $1- \exp ( - 2^{ \delta k} )$. Once $\overline{S}^{1}[\tau ^{1}_{k-1} , \tau^{1}_{k}] \cap {\cal B} (2^{(\frac{m}{3} \vee \frac{31N}{60})+1 }) \neq \emptyset$, the probability that $\overline{S}^{1}$ intersects $\overline{S}^{2}[0, \tau^{2}_{k}]$ until $\tau ^{1}_{k}$ is at least $1- 2^{-\delta k}$. Therefore, by using the strong Markov property, we finish the proof of the lemma for $d=2$.

\end{proof}

\subsubsection{Bounds for $ \frac{m}{3} + 1  \le k \le \frac{21m}{60} $}\label{sub2-1}

We can prove the following lemma by a same idea of Lemma \ref{sub1-1}. So we omit its proof.
\begin{lem}\label{sub2-1}
There exist $\delta > 0$ and $ c < \infty $ such that 
\begin{equation}\label{ineqsub2-1}
P^{w^{1}_{m/3},w^{2}_{m/3}}_{1,2}( B^{1}[T ^{1}_{k}, T^{1}_{N}] \cap B^{2}[T ^{2}_{k}, T^{2}_{N}] = \emptyset , B^{1}[T^{1}_{k-1}, T^{1}_{k}] \cap B^{2}[0,T^{2}_{k}] \neq \emptyset , R_{\frac{m}{3} , N} ) \le c 2^{ -(N- \frac{m}{3}) \xi } 2^{-\delta k} .
\end{equation}
\end{lem}

For $ \frac{m}{3} + 1  \le k \le \frac{21m}{60}$, we have only to show the following lemma.

\begin{lem}\label{sub2-2}
There exist $\delta > 0$ and $ c < \infty $ such that 
\begin{equation}\label{ineqsub2-2}
P^{w^{1}_{m/3},w^{2}_{m/3}}_{1,2}( B^{1}[T ^{1}_{k}, T^{1}_{N}] \cap B^{2}[T ^{2}_{k}, T^{2}_{N}] = \emptyset , B^{1}[T^{1}_{k-1}, T^{1}_{k}] \cap \text{PATH}^{2}_{f} \neq \emptyset , R_{\frac{m}{3} , N} ) \le c 2^{ -(N- \frac{m}{3}) \xi } 2^{-\delta k} .
\end{equation}
\end{lem}

\begin{proof}
We will give a full proof for this lemma. Recall the definition of $\text{PATH}^{2}_{f}$ in \eqref{fatpath}. Let
\begin{equation*}
2\text{PATH}^{2}_{f}= \{ z\in \mathbb{R}^{d} : \text{dist}( z, S^{2}[0,\tau^{2}_{ \frac{m}{3} }] \cup \gamma^{2} ) \le 2^{\frac{11m}{60}+1} \}
\end{equation*}
be the set obtained by letting $\text{PATH}^{2}_{f}$ be fattened twice. Let 
\begin{equation*}
Q^{i}= \{ \sup _{0 \le t \le T^{i}_{k+1} } | \overline{S}^{i}(dt) - B^{i}(t) | \ge 2^{ \frac{31k}{60}} \}
\end{equation*}
and $Q= Q^{1} \cup Q^{2}$. By Proposition \ref{skorohod}, we see that 
\begin{equation*}
P^{w^{1}_{m/3},w^{2}_{m/3}}_{1,2} ( Q ) \le c \exp ( - 2^{\delta k} ),
\end{equation*}
for some $\delta > 0$ and $c< \infty$. Let $\Psi ^{1}(2^{k+1})$, $\Psi^{2}(2^{k+1})$ be the event in Proposition \ref{measurability} for $(B^{1},\overline{S}^{1})$, $(B^{2},\overline{S}^{2})$, respectively. Let $\Psi = \Psi ^{1}(2^{k+1}) \cap \Psi ^{2}(2^{k+1})$. By Proposition \ref{measurability}, 
\begin{equation*}
P^{w^{1}_{m/3},w^{2}_{m/3}}_{1,2} ( \Psi ^{c} ) \le c \exp ( - 2^{\delta k} ),
\end{equation*}
for some $\delta > 0$ and $c< \infty$. Recall that on the event $\Psi$, 
\begin{equation*}
\bigcup_{i=1,2} \{ B^{i}(t) : t \le T^{i}_{k+1} \vee \tau^{i}_{k+1} \} \cup \{ \overline{S} ^{i}(dt) : t \le T^{i}_{k+1} \vee \tau^{i}_{k+1} \}
\end{equation*}
and
\begin{equation*}
\bigcup_{i=1,2} \{ B^{i}(t) : t \ge T^{i}_{k+2} \}
\end{equation*}
are conditionally independent given $B^{1}(T^{1}_{k+2})$ and $B^{2}(T^{2}_{k+2})$. Therefore, 
\begin{align*}
&P^{w^{1}_{m/3},w^{2}_{m/3}}_{1,2} ( B^{1}[T ^{1}_{k}, T^{1}_{N}] \cap B^{2}[T ^{2}_{k}, T^{2}_{N}] = \emptyset , B^{1}[T^{1}_{k-1}, T^{1}_{k}] \cap \text{PATH}^{2}_{f} \neq \emptyset , R_{\frac{m}{3} , N} , Q^{c} , \Psi ) \\
&\le c 2^{-(N-k)\xi } P^{w^{1}_{m/3},w^{2}_{m/3}}_{1,2} ( B^{1}[T^{1}_{k-1}, T^{1}_{k}] \cap \text{PATH}^{2}_{f} \neq \emptyset, R_{\frac{m}{3} , k+1}, Q^{c} , \Psi ).
\end{align*}
From now we will estimate for $P^{w^{1}_{m/3},w^{2}_{m/3}}_{1,2} ( B^{1}[T^{1}_{k-1}, T^{1}_{k}] \cap \text{PATH}^{2}_{f} \neq \emptyset, R_{\frac{m}{3} , k+1}, Q^{c}  )$. 

First, let $d=2$. If $Q^{c}$ holds, then it is easy to see that 
\begin{equation}\label{close}
\frac{\tau ^{1}(2^{k-1} - 2^{\frac{31k}{60}})}{2} \le T^{1}_{k-1} \le T^{1}_{k} \le \frac{\tau ^{1}(2^{k} + 2^{\frac{31k}{60}})}{2}.
\end{equation}
Therefore, on the event $Q^{c} \cap \{ B^{1}[T^{1}_{k-1}, T^{1}_{k}] \cap \text{PATH}^{2}_{f} \neq \emptyset \}$, we have
\begin{equation*}
\text{dist}( \overline{S}^{1}(2t) , S^{2}[0,\tau^{2}_{ \frac{m}{3} }] \cup \gamma^{2} ) \le 2^{\frac{11m}{60}}+ 2^{\frac{31k}{60}} \le 2^{\frac{11m}{60}+1},
\end{equation*}
for some $t \in [\frac{\tau ^{1}(2^{k-1} - 2^{\frac{31k}{60}})}{2} , \frac{\tau ^{1}(2^{k} + 2^{\frac{31k}{60}})}{2}]$. Here the last inequality comes from that $ k \le \frac{21m}{60}$. Hence,
\begin{align}\label{srwclose}
&P^{w^{1}_{m/3},w^{2}_{m/3}}_{1,2} ( B^{1}[T^{1}_{k-1}, T^{1}_{k}] \cap \text{PATH}^{2}_{f} \neq \emptyset, R_{\frac{m}{3} , k+1}, Q^{c}  ) \notag \\
& \le P^{w^{1}_{m/3},w^{2}_{m/3}}_{1,2} ( \overline{S}^{1}[\tau ^{1}(2^{k-1} - 2^{\frac{31k}{60}}), \tau ^{1}(2^{k} + 2^{\frac{31k}{60}}) ] \cap 2\text{PATH}^{2}_{f} \neq \emptyset , R_{\frac{m}{3} , k+1} ).
\end{align}
Let
\begin{equation*}
\sigma = \inf \{ j \ge \tau ^{1}(2^{k-1} - 2^{\frac{31k}{60}}) : \overline{S}^{1}(j) \in 2\text{PATH}^{2}_{f} \}.
\end{equation*}
Then the right hand side in \eqref{srwclose} is bounded above by
\begin{align}\label{stmarkovs1}
E^{w^{2}_{m/3}}_{2} \big( E^{ w^{1}_{m/3}}_{1} \big( & {\bf 1} \{ \overline{S}^{1}[0,\tau ^{1}_{k-2}] \cap \overline{S}^{2}_{k-2} =\emptyset , \sigma \le \tau ^{1}(2^{k} + 2^{\frac{31k}{60}}) \} \notag \\
& \times P^{ \overline{S}^{1}(\sigma ) }( \overline{S}^{1}[0,\tau^{1}_{k+1}] \cap (\overline{S}^{2} [0, \tau^{2}_{k+1}] \cup S[0,\tau^{2}_{\frac{m}{3}}] \cup \gamma^{2}) =\emptyset ) \big) \big).
\end{align}
Since $\overline{S}^{2} [0, \tau^{2}_{k+1}] \cup S[0,\tau^{2}_{\frac{m}{3}}] \cup \gamma^{2}$ is a path from the origin to $\partial {\cal B}(2^{k+1})$, 
\begin{equation*}
\overline{S}^{1}(\sigma ) \in {\cal B}( 2^{\frac{m}{3}}+ 2^{\frac{11m}{60}+1})
\end{equation*}
and
\begin{equation*}
\text{dist} ( \overline{S}^{1}(\sigma ) , S^{2}[0,\tau^{2}_{ \frac{m}{3} }] \cup \gamma^{2} ) \le 2^{\frac{11m}{60}+1},
\end{equation*}
by using Proposition \ref{beurling estimate}, we see that
\begin{equation*}
P^{ \overline{S}^{1}(\sigma ) }( \overline{S}^{1}[0,\tau^{1}_{k+1}] \cap (\overline{S}^{2} [0, \tau^{2}_{k+1}] \cup S[0,\tau^{2}_{\frac{m}{3}}] \cup \gamma^{2}) =\emptyset ) \le c 2^{-\delta k},
\end{equation*}
for some $\delta > 0$ and $c < \infty$. Hence \eqref{stmarkovs1} is bounded above by 
\begin{equation*}
c 2^{-\delta k} P^{w^{1}_{m/3},w^{2}_{m/3}}_{1,2} ( \overline{S}^{1}[0,\tau ^{1}_{k-2}] \cap \overline{S}^{2}_{k-2} =\emptyset ) \le c 2^{-\delta k} 2^{-(k-\frac{m}{3})\xi },
\end{equation*}
and the proof for $d=2$ is finished.

Next we consider for $d=3$. Recall the events $F_{m}$, $G_{m}$ and $H_{m}$ in \eqref{stmarkov}. By \eqref{srwclose}, we need to estimate 
\begin{equation*}
P^{w^{1}_{m/3},w^{2}_{m/3}}_{1,2} ( \overline{S}^{1}[\tau ^{1}(2^{k-1} - 2^{\frac{31k}{60}}), \tau ^{1}(2^{k} + 2^{\frac{31k}{60}}) ] \cap 2\text{PATH}^{2}_{f} \neq \emptyset , R_{\frac{m}{3} , k+1} )
\end{equation*}
on the event $F_{m} \cap G_{m} \cap H_{m}$. For this end, we decompose $2\text{PATH}^{2}_{f}$ into three parts $U_{1},U_{2}$ and $U_{3}$ as follows.
\begin{align*}
&U_{1}= \{ z\in \mathbb{R}^{d} : \text{dist}( z, \gamma^{2} ) \le 2^{\frac{11m}{60}+1} \} \\
&U_{2}= \{ z\in \mathbb{R}^{d} : \text{dist}( z, S^{2}[0,\tau ^{2} (2^{\frac{m}{3}}-2^{\frac{2m}{9}})]  ) \le 2^{\frac{11m}{60}+1} \} \\
&U_{3}= \{ z\in \mathbb{R}^{d} : \text{dist}( z, S^{2}[\tau ^{2} (2^{\frac{m}{3}}-2^{\frac{2m}{9}}),\tau^{2}_{ \frac{m}{3} }]  ) \le 2^{\frac{11m}{60}+1} \}
\end{align*}
Since $\gamma^{2} \in {\cal B} (2^{L})$ and $L \le \frac{m}{10}$, it is easy to see that
\begin{equation*}
P^{w^{1}_{m/3},w^{2}_{m/3}}_{1,2} ( \overline{S}^{1}[\tau ^{1}(2^{k-1} - 2^{\frac{31k}{60}}), \tau ^{1}(2^{k} + 2^{\frac{31k}{60}}) ] \cap U_{1} \neq \emptyset , R_{\frac{m}{3} , k+1} ) \le c 2^{- \delta k} 2^{ -(k-\frac{m}{3}) \xi },
\end{equation*}
for some $\delta > 0$ and $c< \infty$. Since $S^{2}[\tau ^{2} (2^{\frac{m}{3}}-2^{\frac{2m}{9}}),\tau^{2}_{ \frac{m}{3} }] \subset {\cal B} ( S^{2}(\tau ^{2} (2^{\frac{m}{3}}-2^{\frac{2m}{9}})), 2^{\frac{m}{4}})$ on the event $G_{m}$, we see that $U_{3} \subset {\cal B} ( S^{2}(\tau ^{2} (2^{\frac{m}{3}}-2^{\frac{2m}{9}})), 2^{\frac{m}{4}+1})$. Therefore $\overline{S}^{1}[\tau ^{1}((2^{k-1} - 2^{\frac{31k}{60}}) \vee 2^{\frac{m}{3}}), \tau ^{1}(2^{k} + 2^{\frac{31k}{60}}) ] \cap U_{3} \neq \emptyset$ implies that
\begin{equation}\label{s1enterball}
\overline{S}^{1}[\tau ^{1}((2^{k-1} - 2^{\frac{31k}{60}}) \vee 2^{\frac{m}{3}}), \tau ^{1}(2^{k} + 2^{\frac{31k}{60}}) ] \cap {\cal B} ( S^{2}(\tau ^{2} (2^{\frac{m}{3}}-2^{\frac{2m}{9}})), 2^{\frac{m}{4}+1}).
\end{equation}
However,
\begin{equation*}
|\overline{S}^{1}(\tau ^{1}((2^{k-1} - 2^{\frac{31k}{60}}) \vee 2^{\frac{m}{3}}) - S^{2}(\tau ^{2} (2^{\frac{m}{3}}-2^{\frac{2m}{9}}))| \ge 2^{\frac{0.99m}{3}},
\end{equation*}
on the event $F_{m}$. So the probability of \eqref{s1enterball} is bounded above by $c 2^{-\frac{m}{24}}$ for some $c < \infty$. Using the strong Markov property, 
\begin{equation*}
P^{w^{1}_{m/3},w^{2}_{m/3}}_{1,2} ( \overline{S}^{1}[\tau ^{1}((2^{k-1} - 2^{\frac{31k}{60}}) \vee 2^{\frac{m}{3}}), \tau ^{1}(2^{k} + 2^{\frac{31k}{60}}) ] \cap U_{3} \neq \emptyset , R_{\frac{m}{3} , k+1} ) \le c 2^{- \delta k} 2^{ -(k-\frac{m}{3}) \xi },
\end{equation*}
for some $\delta > 0$ and $c <\infty$. Finally we consider for $U_{2}$. Let
\begin{equation*}
\sigma^{\sharp}= \inf \{ j \ge \tau ^{1}(2^{k-1} - 2^{\frac{31k}{60}}) : \overline{S}^{1} (j) \in U_{2} \}.
\end{equation*}
Then by the strong Markov property,
\begin{align}\label{hmap}
&P^{w^{1}_{m/3},w^{2}_{m/3}}_{1,2} ( \overline{S}^{1}[\tau ^{1}(2^{k-1} - 2^{\frac{31k}{60}}), \tau ^{1}(2^{k} + 2^{\frac{31k}{60}}) ] \cap U_{2} \neq \emptyset , R_{\frac{m}{3} , k+1} ) \notag \\
&\le E^{w^{2}_{m/3}}_{2} \big( E^{w^{1}_{m/3}}_{1} \big( {\bf 1} \{ R_{\frac{m}{3} , k-2} , \sigma^{\sharp} \le \tau ^{1}(2^{k} + 2^{\frac{31k}{60}}) \} \notag \\
& \times P^{\overline{S}^{1}(\sigma^{\sharp})}_{1}( \overline{S}^{1}[0, \tau^{1}_{k+1}] \cap S^{2}[0,\tau^{2}_{\frac{m}{3}}] = \emptyset ) \big) \big).
\end{align}
Note that $P^{\overline{S}^{1}(\sigma^{\sharp})}_{1}( \overline{S}^{1}[0, \tau^{1}_{k+1}] \cap S^{2}[0,\tau^{2}_{\frac{m}{3}}] = \emptyset ) \le Z^{2}_{m}$. Hence on the event $H_{m}$, the right hand side of \eqref{hmap} can be bounded above by $c 2^{-\delta k}2^{-(k-\frac{m}{3})\xi}$ for some $\delta > 0$ and $c < \infty$, and the lemma is proved.

\end{proof}

\subsubsection{Bounds for $ k=N-2,N-1,$ and $k=N$}\label{sub3-1}

Again, we will only consider for $k=N$ as in Section \ref{sub3}. Other cases can be estimated by a similar argument given below.

\begin{lem}\label{ineqsub3-1}
There exist $\delta > 0$ and $ c < \infty $ such that 
\begin{equation}\label{ineqsub3-1}
P^{w^{1}_{m/3},w^{2}_{m/3}}_{1,2}(  B^{1}[T^{1}_{N-1}, T^{1}_{N}] \cap (B^{2}[0, T^{2}_{N}] \cup \text{PATH}^{2}_{f}) \neq \emptyset , R_{\frac{m}{3} , N} ) \le c 2^{ -(N- \frac{m}{3}) \xi } 2^{-\delta N} .
\end{equation}
\end{lem}

\begin{proof}
We will sketch the proof.
First we consider the following probability,
\begin{equation}\label{ineqsub3-1-1}
P^{w^{1}_{m/3},w^{2}_{m/3}}_{1,2}(  B^{1}[T^{1}_{N-1}, T^{1}_{N}] \cap  \text{PATH}^{2}_{f} \neq \emptyset , R_{\frac{m}{3} , N} ). 
\end{equation}
The probability that $B^{1}$ enters ${\cal B}_{\frac{m}{3}}$ during $[T^{1}_{N-1}, T^{1}_{N}]$ is bounded above by $c 2^{-(N-\frac{m}{3})}$ for $d=3$. Therefore by using the strong Markov property, we see that \eqref{ineqsub3-1-1} can be bounded above by $c 2^{-(N-\frac{m}{3})} 2^{-(N-\frac{m}{3})\xi } $. Since $N \ge m$, we have $2^{-(N-\frac{m}{3})} 2^{-(N-\frac{m}{3})\xi } \le 2^{-\frac{N}{2}} 2^{-(N-\frac{m}{3})\xi }$. 

For $d=2$, we use Proposition \ref{beurling estimate} as follows. Assume $B^{1}$ enters ${\cal B}_{\frac{m}{3}}$ during $[T^{1}_{N-1}, T^{1}_{N}]$. Then by a similar argument given in the proof of Lemma \ref{5step}, $\overline{S}^{1}$ also enters ${\cal B}(2^{\frac{2N}{3}})$ during $[\tau ^{1}_{N-1}, \tau ^{1}_{N}]$ with probability at least $1- 2^{-\frac{N}{6}}$. After $\overline{S}^{1}$ enters ${\cal B}(2^{\frac{2N}{3}})$, it follows from Proposition \ref{beurling estimate} that the probability that $\overline{S}^{1}$ does not intersect $\overline{S}^{2}[0, \tau^{2}_{N}] $ until it reaches $\partial {\cal B} (2^{N})$ is bounded above by $c 2^{-\frac{N}{6}}$. Combining these estimate, we see that \eqref{ineqsub3-1-1} can be bounded above by $c 2^{ {-\frac{N}{6}}} 2^{-(N-\frac{m}{3})\xi }$ for $d=2$. 

Therefore, in order to show \eqref{ineqsub3-1}, we need to estimate the following probability,
\begin{equation}\label{ineqsub3-1-2}
P^{w^{1}_{m/3},w^{2}_{m/3}}_{1,2}(  B^{1}[T^{1}_{N-1}, T^{1}_{N}] \cap B^{2}[0, T^{2}_{N}]  \neq \emptyset , R_{\frac{m}{3} , N} ).
\end{equation}
By similar arguments as in Lemma \ref{6stepsub1} and Remark \ref{rem1}, we have
\begin{align*}
&P^{w^{1}_{m/3},w^{2}_{m/3}}_{1,2}( B^{1}[T^{1}_{N-1}, T^{1}_{N}] \cap B^{2}[T^{2}(2^{N}-2^{\frac{2N}{3}}) , T^{2}_{N}] \neq \emptyset , R_{\frac{m}{3} , N} ) \le c 2^{-\delta N} 2^{-(N-\frac{m}{3})\xi } \\
&P^{w^{1}_{m/3},w^{2}_{m/3}}_{1,2}( B^{1}[T^{1}(2^{N}-2^{\frac{2N}{3}}) , T^{1}_{N}] \cap B^{2}[0 , T^{2}_{N}] \neq \emptyset , R_{\frac{m}{3} , N} ) \le c 2^{-\delta N} 2^{-(N-\frac{m}{3})\xi },
\end{align*}
for some $\delta > 0$ and $c < \infty$. So, assume $B^{1}[T^{1}_{N-1} , T^{1}(2^{N}-2^{\frac{2N}{3}}) ] \cap B^{2}[0 , T^{2}(2^{N}-2^{\frac{2N}{3}})] \neq \emptyset$. Then by Proposition \ref{skorohod}, 
\begin{equation}\label{srwbmgetsclose}
\text{dist} \big( \overline{S}^{1}[\tau^{1}(2^{N-1}-2^{\frac{31N}{60}}), \tau^{1}(2^{N}-2^{\frac{2N}{3}}+2^{\frac{31N}{60}}) ] , \overline{S}^{2}[0,\tau ^{2}(2^{N}-2^{\frac{2N}{3}}+2^{\frac{31N}{60}})] \big) \le 2^{\frac{31N}{60}},
\end{equation}
with probability at least $1- \exp ( -2^{\delta N})$. Then by modifying the proof of Lemma \ref{6stepsub1}, we see that 
\begin{equation*}
P^{w^{1}_{m/3},w^{2}_{m/3}}_{1,2} ( \{ \eqref{srwbmgetsclose} \text{ holds } \} \cap R_{\frac{m}{3} , N} ) \le c 2^{-\delta N} 2^{-(N-\frac{m}{3})\xi },
\end{equation*}
which gives the proof of the lemma.

\end{proof} 

\subsubsection{Conclusion \ Upper Bound}
Combining estimates obtained in subsections \ref{sub1-1}, \ref{sub2-1} and \ref{sub3-1} with \eqref{upperbound} and Proposition \ref{lowerbound}, we have the following.

\begin{prop}\label{upper-bound}
There exist $\delta > 0$ and $c < \infty$ such that 
\begin{equation}\label{upperboundesti}
|P^{w^{1}_{m/3},w^{2}_{m/3}}_{1,2}(R_{\frac{m}{3} , N}) - P^{w^{1}_{m/3},w^{2}_{m/3}}_{1,2}( J_{m,N})| \le  c 2^{-\delta m}2^{-(N-\frac{m}{3})\xi},
\end{equation}
on the event $F_{m} \cap G_{m} \cap H_{m}$.
\end{prop}

\section{Proof of Main Theorem}
\subsection{Cauchy sequence} 

Fix $L \in \mathbb{N}$ and $\overline{\gamma} \in \Gamma _{L}$. For $m \ge 10L$, define
\begin{equation}\label{cauchy}
Q(m, \overline{\gamma} ) = 2^{(m-L)\xi } P(A_{m}( \overline{\gamma} )).
\end{equation}
We will show the following Theorem.

\begin{thm}\label{cauchy seq}
There exist $ \delta > 0$ and $c < \infty$ depending only on the dimension such that for all $n \ge m \ge 10L$, we have
\begin{align}\label{cauchyesti}
| Q(m, \overline{\gamma} ) - Q(n, \overline{\gamma} ) | \le c 2^{-\delta \sqrt{m}}, \ \ \text{ for } d=2, \\
| Q(m, \overline{\gamma} ) - Q(n, \overline{\gamma} ) | \le c 2^{-\delta m}, \ \ \ \text { for } d=3.
\end{align}
\end{thm}

\begin{proof}
Fix $L \in \mathbb{N}$, $\overline{\gamma} \in \Gamma _{L}$ and $n \ge m \ge 10L$. Then
\begin{align*}
&| Q(m, \overline{\gamma} ) - Q(n, \overline{\gamma} ) | \\
&\le | 2^{(m-L)\xi } P ( A_{m}( \overline{\gamma} ) \cap F_{m} \cap G_{m} \cap H_{m} ) - 2^{(n-L)\xi }  P ( A_{n}( \overline{\gamma} ) \cap F_{m} \cap G_{m} \cap H_{m} ) | + c 2^{-\delta m},
\end{align*}
for some $\delta > 0$ and $ c < \infty$. By the strong Markov property,
\begin{align}\label{bunri}
&| 2^{(m-L)\xi } P ( A_{m}( \overline{\gamma} ) \cap F_{m} \cap G_{m} \cap H_{m} ) - 2^{(n-L)\xi }  P ( A_{n}( \overline{\gamma} ) \cap F_{m} \cap G_{m} \cap H_{m} ) | \notag \\
&=| 2^{(\frac{m}{3}-L)\xi } E \big( {\bf 1} \{ A_{\frac{m}{3}}( \overline{\gamma} ) \cap F_{m} \cap G_{m} \cap H_{m} \} 2^{ \frac{2m}{3} \xi} P^{w^{1}_{m/3},w^{2}_{m/3}}_{1,2} ( R_{\frac{m}{3} , m} ) \big) \notag \\
 &- 2^{(\frac{m}{3}-L)\xi } E \big( {\bf 1} \{ A_{\frac{m}{3}}( \overline{\gamma} ) \cap F_{m} \cap G_{m} \cap H_{m} \} 2^{ (n-\frac{m}{3}) \xi } P^{w^{1}_{m/3},w^{2}_{m/3}}_{1,2} ( R_{\frac{m}{3} , n} ) \big) |
\end{align}
By Proposition \ref{upper-bound}, we have 
\begin{align*}
&|P^{w^{1}_{m/3},w^{2}_{m/3}}_{1,2} ( R_{\frac{m}{3} , m} ) \big) - P^{w^{1}_{m/3},w^{2}_{m/3}}_{1,2}( J_{m,m}) | \le c 2^{-\delta m}2^{-\frac{2m}{3}\xi} \\
&|P^{w^{1}_{m/3},w^{2}_{m/3}}_{1,2}(R_{\frac{m}{3} , n}) - P^{w^{1}_{m/3},w^{2}_{m/3}}_{1,2}( J_{m,n})| \le  c 2^{-\delta m}2^{-(n-\frac{m}{3})\xi},
\end{align*}
on the event $F_{m} \cap G_{m} \cap H_{m}$. Therefore, the right hand side of \eqref{bunri} is bounded above by
\begin{align*}
&2^{(\frac{m}{3}-L)\xi} E \big( {\bf 1} \{ V^{m} \} 2^{ \frac{2m}{3} \xi} |P^{w^{1}_{m/3},w^{2}_{m/3}}_{1,2} ( R_{\frac{m}{3} , m} )- P^{w^{1}_{m/3},w^{2}_{m/3}}_{1,2}( J_{m,m})| \big) \\
&+ 2^{(\frac{m}{3}-L)\xi} E \big( {\bf 1} \{ V^{m} \} | 2^{ \frac{2m}{3} \xi} P^{w^{1}_{m/3},w^{2}_{m/3}}_{1,2}( J_{m,m}) - 2^{ (n-\frac{m}{3}) \xi } P^{w^{1}_{m/3},w^{2}_{m/3}}_{1,2}( J_{m,n}) | \big) \\
&+ 2^{(\frac{m}{3}-L)\xi} E \big( {\bf 1} \{ V^{m} \} 2^{ (n-\frac{m}{3}) \xi} |P^{w^{1}_{m/3},w^{2}_{m/3}}_{1,2} ( R_{\frac{m}{3} , n} )- P^{w^{1}_{m/3},w^{2}_{m/3}}_{1,2}( J_{m,n})| \big) \\
&\le 2^{(\frac{m}{3}-L)\xi} E \big( {\bf 1} \{ V^{m} \} | 2^{ \frac{2m}{3} \xi} P^{w^{1}_{m/3},w^{2}_{m/3}}_{1,2}( J_{m,m}) - 2^{ (n-\frac{m}{3}) \xi } P^{w^{1}_{m/3},w^{2}_{m/3}}_{1,2}( J_{m,n}) | \big) \\
&+ c 2^{-\delta m} 2^{(\frac{m}{3}-L)\xi} P(V^{m}),
\end{align*}
where $V^{m} = A_{\frac{m}{3}}( \overline{\gamma} ) \cap F_{m} \cap G_{m} \cap H_{m}$. By Proposition \ref{convergencethm}, 
\begin{equation*}
| 2^{ \frac{2m}{3} \xi} P^{w^{1}_{m/3},w^{2}_{m/3}}_{1,2}( J_{m,m}) - 2^{ (n-\frac{m}{3}) \xi } P^{w^{1}_{m/3},w^{2}_{m/3}}_{1,2}( J_{m,n}) | \le c 2^{-\delta \sqrt{m}},
\end{equation*}
for $d=2$ and 
\begin{equation*}
| 2^{ \frac{2m}{3} \xi} P^{w^{1}_{m/3},w^{2}_{m/3}}_{1,2}( J_{m,m}) - 2^{ (n-\frac{m}{3}) \xi } P^{w^{1}_{m/3},w^{2}_{m/3}}_{1,2}( J_{m,n}) | \le c 2^{-\delta m},
\end{equation*}
for $d=3$ on the event $V^{m}$. Hence 
\begin{align*}
&2^{(\frac{m}{3}-L)\xi} E \big( {\bf 1} \{ V^{m} \} | 2^{ \frac{2m}{3} \xi} P^{w^{1}_{m/3},w^{2}_{m/3}}_{1,2}( J_{m,m}) - 2^{ (n-\frac{m}{3}) \xi } P^{w^{1}_{m/3},w^{2}_{m/3}}_{1,2}( J_{m,n}) | \big) \\
&\le c  2^{-\delta m^{\frac{d}{2}-\frac{1}{2}}} 2^{(\frac{m}{3}-L)\xi} P ( V^{m} ).
\end{align*}
Finally, by the strong Markov property, 
\begin{equation*}
P ( V^{m} ) \le P( A_{L+1}( \overline{\gamma} ) )c 2^{ -( \frac{m}{3} - L)\xi },
\end{equation*}
and the proof is finished.

\end{proof}
From Theorem \ref{cauchy seq}, we get the following corollary immediately.

\begin{cor}\label{coro}
There exist $\delta > 0$ and $ c < \infty$ such that the following holds. For each $L \in \mathbb{N}$ and $\overline{\gamma} \in \Gamma _{L}$, there exists $Q( \overline{\gamma} ) \in (0,1)$ such that 
\begin{align}\label{coro1}
&\lim _{m \rightarrow \infty} Q(m, \overline{\gamma} ) = Q( \overline{\gamma} ) \\
&| Q(m, \overline{\gamma} ) - Q( \overline{\gamma} ) | \le c 2^{-\delta m^{\frac{d}{2}-\frac{1}{2}}}.
\end{align}
Especially, there exists a $\alpha \in (0,1)$ such that
\begin{align}\label{coro2}
P \big( (S^{1}[0, \tau^{1}_{n}], S^{2}[0,\tau^{2}_{n}]) \in \Gamma_{n} \big) \sim \alpha 2^{-n\xi} \\
|P \big( (S^{1}[0, \tau^{1}_{n}], S^{2}[0,\tau^{2}_{n}]) \in \Gamma_{n} \big) 2^{n\xi} - \alpha | \le c 2^{-\delta n^{\frac{d}{2}-\frac{1}{2}}}.
\end{align}
\end{cor}

\begin{cor}\label{coro-2}
There exist $\delta > 0$ and $ c < \infty$ such that the following holds. For each $L \in \mathbb{N}$ and $\overline{\gamma} \in \Gamma _{L}$, the limit
\begin{equation}\label{existmeas}
\lim _{N \rightarrow \infty } P( (S^{1}[0, \tau^{1}_{L}], S^{1}[0, \tau^{2}_{L}] ) = \overline{\gamma} \  | \  (S^{1}[0, \tau^{1}_{N}], S^{1}[0, \tau^{2}_{N}] ) \in \Gamma _{N} )
\end{equation}
exists. If we write $P^{\sharp}( \overline{\gamma} )$ for the limit, then 
\begin{equation}\label{speedofconv}
| P( (S^{1}[0, \tau^{1}_{L}], S^{1}[0, \tau^{2}_{L}] ) = \overline{\gamma} \ | \ (S^{1}[0, \tau^{1}_{N}], S^{1}[0, \tau^{2}_{N}] ) \in \Gamma _{N} ) - P^{\sharp}( \overline{\gamma} ) | \le c 2^{-\delta N^{\frac{d}{2}-\frac{1}{2}}}.
\end{equation}
\end{cor}

\begin{proof}
Fix $L \in \mathbb{N}$ and $\overline{\gamma} \in \Gamma _{L}$. Let
\begin{equation*}
p( \overline{\gamma}) = P( (S^{1}[0, \tau^{1}_{L}], S^{1}[0, \tau^{2}_{L}] ) = \overline{\gamma} ).
\end{equation*}
Then 
\begin{align*}
&P( (S^{1}[0, \tau^{1}_{L}], S^{1}[0, \tau^{2}_{L}] ) = \overline{\gamma} \  | \  (S^{1}[0, \tau^{1}_{N}], S^{1}[0, \tau^{2}_{N}] ) \in \Gamma _{N} ) \\
&=\frac{ p( \overline{\gamma}) P ( A_{N}( \overline{\gamma} )) }{ P ( (S^{1}[0, \tau^{1}_{N}], S^{1}[0, \tau^{2}_{N}] ) \in \Gamma _{N} )}
\end{align*}
By Corollary \ref{coro}, letting $P^{\sharp}(\overline{\gamma} )$ be
\begin{equation*}
\frac{p( \overline{\gamma} ) 2^{L\xi} Q( \overline{\gamma} ) }{\alpha},
\end{equation*}
the proof is finished.

\end{proof}

\begin{rem}\label{hokan}
In order to simplify the notations, all results above were stated for the first hitting time of $\partial {\cal B}(2^{N})$ instead of $\partial {\cal B} (N)$. However there is no essential difference between them and similar arguments also work for the latter case. Since it is easy to extend above results to the hitting time of $\partial {\cal B} (N)$, we leave the details to the reader.

\end{rem}

{\bf Acknowledgement.} The author wishes to express his sincere gratitude to Greg Lawler for his guidance and helpful discussions. I also thank Takashi Kumagai for careful reading of the early version of the manuscript.

\end{document}